\documentclass[12pt,reqno]{amsart}

\usepackage[T1]{fontenc}
\usepackage{enumitem}
\usepackage{hyperref}
\hypersetup{
  colorlinks   = true,
  citecolor    = blue,
  linkcolor    = blue
}

\usepackage{amsmath,amsfonts,amsthm,amssymb,color,tikz}
\usepackage{mathrsfs}   
\usepackage{comment}
\usepackage{bbm}

\usepackage{pdfsync}
\usepackage[font={scriptsize}]{caption}

\usepackage[left=1in, right=1in, top=1.1in,bottom=1.1in]{geometry}
\usepackage{yhmath} % for very wide hat

\newcommand{\dom}{\mbox{Dom}}

\newcommand{\1}{{\bf 1}}

 \newcommand{\Heis}{{\mathbf{H}^{n}}}   
  \newcommand{\Heissq}{({\mathbf{H}^{n})^2}}   
 \newcommand{\bW}{\mathbf{W}}
 
 \newcommand{\bE}{\mathbf{E}}
 \newcommand{\BH}{\mathcal{B}}

\newcommand{\E}{\mathbb E}
\newcommand{\R}{\mathbb R}
\newcommand{\N}{\mathbb N}

 \newcommand{\bww}{\mathbf{W}}
  
    \newcommand{\bX}{\mathbf{X}}

\newcommand{\cb}{\mathcal B}

\newcommand{\cd}{\mathcal D}

\newcommand{\cg}{\mathcal G}
\newcommand{\ch}{\mathcal H}

\newcommand{\cw}{\mathcal W}

\newcommand{\al}{\alpha}

\newcommand{\la}{\lambda}
\newcommand{\laa}{\Lambda}

\newcommand{\lp}{\left(}
\newcommand{\rp}{\right)}
\newcommand{\lc}{\left[}
\newcommand{\rc}{\right]}
\newcommand{\lcl}{\left\{}
\newcommand{\rcl}{\right\}}

\newcommand{\m}{k}

\newtheorem{theorem}{Theorem}[section]

\newtheorem{definition}[theorem]{Definition}

\newtheorem{lemma}[theorem]{Lemma}

\newtheorem{proposition}[theorem]{Proposition}

\theoremstyle{remark}
\newtheorem{remark}[theorem]{Remark}

\newtheorem{notation}[theorem]{Notation}
\theoremstyle{remark}

\newcommand{\bean}{\begin{eqnarray*}}
\newcommand{\eean}{\end{eqnarray*}}
\newcommand{\ben}{\begin{enumerate}}
\newcommand{\een}{\end{enumerate}}
\newcommand{\beq}{\begin{equation}}
\newcommand{\eeq}{\end{equation}}

\begin{document}

\title[PAM on Heisenberg group]{Parabolic Anderson model on Heisenberg groups:\\ the It\^o setting}

\author[{F. Baudoin \and C. Ouyang \and S. Tindel \and J. Wang}]
{Fabrice Baudoin \and Cheng Ouyang \and Samy Tindel \and Jing Wang}

\address{Fabrice Baudoin: Department of Mathematics,
University of Connecticut,
341 Mansfield Road,
Storrs, CT  06269, United States.}
\email{fabrice.baudoin@uconn.edu}
\thanks{F. Baudoin is supported by the NSF grant  DMS-1901315}

\address{Cheng Ouyang: Department of Mathematics, 
Statistics and Computer Science, University of Illinois at Chicago, Chicago, 
United States.}
\email{couyang@uic.edu}
\thanks{C. Ouyang is supported in part by Simons grant \#851792}

\address{Samy Tindel: Department of Mathematics, Purdue University, 150 N. University Street, West Lafayette, IN 47907, United States.}
\email{stindel@purdue.edu}
\thanks{S. Tindel is supported by the NSF grant  DMS-1952966.}

\address{Jing Wang: Department of Mathematics, Department of Statistics, Purdue University, 150 N. University Street, West Lafayette, IN 47907, United States.}
\email{jingwang@purdue.edu}
\thanks{J. Wang is supported by the NSF grant  DMS-1855523}

\begin{abstract}
In this note we focus our attention on a stochastic heat equation defined on the Heisenberg group $\Heis$ of order $n$. This equation is written as $\partial_t u=\frac{1}{2}\Delta u+u\dot{W}_\alpha$, where $\Delta$ is the hypoelliptic Laplacian on $\Heis$ and $\{\dot{W}_\alpha; \alpha>0\}$ is a family of Gaussian space-time noises which are white in time and have a covariance structure generated by $(-\Delta)^{-\alpha}$ in space. Our aim is threefold: (i) Give a proper description of the noise $W_\alpha$; (ii) Prove that one can solve the stochastic heat equation in the It\^{o} sense as soon as $\alpha>\frac{n}{2}$; (iii) Give some basic moment estimates for the solution $u(t,x)$.
\end{abstract}

\maketitle

\tableofcontents

%{\color{red}{To do: Unify the notation $\Delta$, $\Delta_\mathcal{H}$, $\Delta$}}

\section{Introduction}

Many fascinating links between parabolic Anderson models \cite{Da, Dalang-Quer, HHNT, HN09, PZ}, KPZ equation \cite{Ha-14} and polymer models \cite{CY, La10, MRT, RT} on $\R^d$ or $\mathbb{Z}^d$ have been recently established. Challenging properties such as intermittency \cite{HHNT,Kh} and localization \cite{Ko-bk,La10} have also been derived in various contexts. Overall one can argue that those models are now fairly well understood on flat spaces, although the body of literature on the topic is still steadily growing. 

In this article we start a series of studies aiming at investigating if the geometry of the underlying space has a chance to influence  the intermittency and/or localization features of parabolic Anderson models. In fact previous studies \cite{TV} tend to show that compact manifolds will not yield behaviors which are significantly different from the flat situation, as far as PAM models are concerned. On the other hand, the recent contribution \cite{CSZ} exhibits more substantial changes on non-compact discrete manifolds such as infinite graphs. We have thus decided to turn our attention to a class of non compact sub-Riemannian manifolds for which explicit computations are still available, namely the family $\{\Heis, n\geq1\}$ of Heisenberg groups based on $\R^n$. 

More specifically let $\Heis$ be the space $\R^{2n}\times\R$ equipped with the following product defined for $(a,b,c), (a', b', c')\in\R^{n}\times\R^n\times\R$: 
\[
(a,b,c)\star (a',b',c'):=\left(a+a', b+b', c+c'+2\omega((a,b),(a',b')) \right),\quad (a,b), \, (a',b')\in \R^{2n} \, ,
\]
where $\omega$ is the symplectic form $\omega((a,b),(a',b'))=\sum_{i=1}^n a_i' b_i-a_ib_i'$. Basic facts about Heisenberg groups will be recalled in Section \ref{sec:preliminaries}. At this point let us just recall that a subelliptic Laplace operator $\Delta$ can be defined on $\Heis$. With this operator in hand, we consider the following linear equation on $\R_+\times\Heis$:
\begin{equation}\label{eq:pam}
\partial_{t} u(t,q) = \frac{1}{2}\Delta u(t,q) + u(t,q) \, \dot{W}(t,q),
\end{equation}
interpreted in the It\^{o} sense. We focus on a proper definition, existence-uniquness result and basic moment estimates for equation \eqref{eq:pam}. Our findings can be summarized as follows:

\begin{enumerate}[wide, labelwidth=!, labelindent=0pt, label= \textbf{(\roman*)}]
\setlength\itemsep{.05in}

\item
 A substantial part of our effort is dedicated to properly define and study a natural class of space-time Gaussian noise $\{W_\alpha; 0<\alpha<\frac{n}{2}+1\}$ on $\R_{+}\times\Heis$. For the sake of conciseness, we will restrict our study here to noises which are white in time. As far as the space variable is concerned, our noisy inputs will deviate slightly from fractional Brownian motion or Riesz-type noises which are usually considered for stochastic PDEs on $\R^d$ (see e.g. \cite{ HHNT, HN09}). Namely for a fixed $\alpha\in(0,\frac{n}{2}+1)$, the noise $W_\alpha$ is defined as a centered Gaussian family $\{\bW_\alpha(\phi); \phi\in \mathcal{H}\}$, with a covariance function of the form 
\begin{align}\label{intro: eq-cov-0<s<1}
\bE\left[ \bW_{\!\al}(\varphi)\bW_{\!\al}(\psi)\right]
&
=\int_{\R_+}\int_{\Heis} \lc (-\Delta)^{-\alpha} \varphi\rc \!(t,q_1) \, \lc (-\Delta)^{-\alpha} \psi\rc \!(t,q_1) \, d\mu(q_1) dt,
\end{align}
where the Hilbert space $\mathcal{H}$ is based on a proper Sobolev-type structure and where $\mu$ stands for the Haar measure on $\Heis$. Although this kind of noise might be related to Riesz-type noises due to the properties of the kernel related to $(-\Delta)^{-\alpha}$, our definition is in fact inspired by \cite{LSSW} and references therein (with motivations rooted in the analysis of Gaussian free fields and log-correlated processes).

\item 
Once the noise $\bW_\alpha$ is defined, we  investigate basic properties for this process (such as invariance with respect to dilations, rotations and left translations on $\Heis$). We also separate a regime $0<\alpha<\frac{n+1}{2}$ for which $\bW_\alpha(t,\cdot)$ is distribution-valued, from the situation $\frac{n+1}{2}<\alpha<\frac{n}{2}+1$ which yields a point-wise definition of $x\mapsto\bW_\alpha(t,x)$ as a function. Notice that the case $\alpha=0$ corresponds to a white noise on $\R_+\times\Heis$, while $\bW_\alpha$ with $\alpha=\frac{n+1}{2}$ is a log-correlated field and a parameter $\alpha>\frac{n}{2}+1$ means that $x\mapsto\bW_\alpha(t,x)$ is differentiable.  We do not take those cases into consideration for the sake of conciseness. A Besov space analysis of the field $\bW_\alpha$ is postponed to a subsequent paper.

\item The field $\bW_\alpha$ serves as a noisy input for equation \eqref{eq:pam}. In Section \ref{existence and uniqueness} we build a stochastic integration theory inspired by \cite{Da, Dalang-Quer}, which enables the definition of a random field mild solution to the stochastic heat equation. Then we handle existence results thanks to two different methods:
\begin{enumerate}[wide, labelwidth=!, labelindent=0pt, label= {(\arabic*)}]
\setlength\itemsep{.05in}

\item
When $\alpha\in(\frac{n}{2}, \frac{n+1}{2})$, one can apply the general results from \cite{PT} in order to get existence and uniqueness thanks to It\^{o} type calculus considerations. Notice that this method only applies to the case when $\bW_{\al}(t,\cdot)$ is distribution-valued.

\item 
We also revisit the existence-uniqueness problem through chaos expansions. This presents several advantages: it allows us to treat the case $\alpha\in (\frac{n+1}{2}, \frac{n}{2}+1)$, and it also yields necessary conditions on $\alpha$.
\end{enumerate}

\item
Another advantage of the chaos expansion method alluded to above is that it potentially leads to sharp estimates for moments of the solution. This is often achieved through Feynman-Kac representations. In this paper we restrict our computations to a basic exponential type estimates of the form
$$c_1e^{c_2t}\leq \E\left[(u(t,x))^2\right]\leq c_3e^{c_4t},$$
where $c_1,\ldots, c_4$ are unspecified positive constants. However, it will be clear from the our considerations that our setting is amenable to precise moments asymptotics similar to \cite{Ch14, Ch17,HHNT}.
\end{enumerate}

\noindent
Overall our paper has to be seen as a contribution setting up the basics of a full stochastic analysis for the parabolic Anderson model on $\Heis$. Our considerations are based on Malliavin calculus and stochastic integration for random fields. The main novelty in our method is to incorporate advanced tools of analysis on Heisenberg groups (reverse Poincar\'{e} inequality, small ball probabilities for the Brownian motion to quote a few) into this framework. In particular, we made an extensive use of the projective approach to Fourier analysis on $\Heis$, recently advocated in \cite{BCD}. This allows us to express many of our conditions in a neat and explicit way.

As mentioned above, one of our main goals in this project is to track down how the geometry for spaces of the form $\Heis$ can influence the global behavior of a system like the stochastic heat equation. As far as existence and uniqueness in the It\^{o} setting is concerned, one should compare our noise $\bW_\alpha$ to the closest family of Gaussian noises considered on $\R^d$ in the SPDE literature. Arguably this family is given by Bessel kernels (see e.g. \cite[Example 2.4]{HHNT}), which correspond to a covariance structure analogous to \eqref{intro: eq-cov-0<s<1}:
\[ 
\E\left[\left|\bW^{\R^d}_\alpha(\varphi)\right|^2\right]=\left\|(I-\Delta)^{-\alpha}\varphi\right\|^2_{L^2(\R^d)}.
\]
Observe that in Dalang's terminology, the spectral measure of $\bW_\alpha$ is given by $\nu(d\xi)=(1+|\xi|^2)^{2\alpha} \, d\xi$. It is then proved that one can solve the stochastic heat equation on $\R^d$ driven by $\bW^{\R^d}_\alpha$ if and only if $\alpha>\frac{d}{4}-\frac{1}{2}$. If one wishes to compare this condition in the flat space $\R^d$ with our framework in $\Heis$, we must consider the topological dimension of $\Heis$ given by $d=2n+1$. The condition $\alpha>\frac{n}{2}$ stated above for existence-uniqueness should thus be read as $\alpha>\frac{d}{4}-\frac{1}{4}$, as opposed to the $\alpha>\frac{d}{4}-\frac{1}{2}$ mentioned for the $\R^d$ case. Importantly enough, this discrepancy should be explained as an effect of the degeneracy coming from the sub-Riemannian structure of $\Heis$. Indeed, if one replaces the topological dimension $d=2n+1$ by the Hausdorff dimension  $Q=2(n+1)$ (see Remark \ref{Rmk: Hurst parameter} below for further details, as well as a comparison with the fractional noises in \cite{LSSW}), then our condition $\alpha>\frac{n}{2}$ reads $\alpha>\frac{Q}{4}-\frac{1}{2}$. We thus go back to the condition on $\bW_\alpha$ alluded to in the flat case. Therefore it is clearly seen that some of the sub-Riemannian geometric structure of $\Heis$ do affect the behavior of stochastic heat equations in a crucial way. 

As mentioned above, our study of geometric features in the parabolic Anderson model calls for further developments. Among those, let us highlight the following natural questions:
\begin{enumerate}[wide, labelwidth=!, labelindent=0pt, label= \textbf{(\alph*)}]
\setlength\itemsep{.05in}

\item
Pathwise definition of stochastic heat equations in the stratonovich sense, which relies on the prior introduction of weighted Besov spaces on $\Heis$.

\item Higher order expansions and renormalization techniques for small values of $\alpha$ in the noise $\bW_\alpha$. This generalization requires cumbersome regularity structures techniques.

\item Definition of polymer measures related to equation \eqref{eq:pam}. Then one should study related exponents and disorder regimes affected by the geometry of $\Heis$. Related to this question, one would also like to derive sharp asymptotics for the moments of $u(t,x)$ as $t\to\infty$.

\item Extensions to more general underlying spaces. In particular, we believe that fractals might give a wide variety of exotic exponents for both PAM and polymer measures. 

\end{enumerate}

\noindent
We plan to tackle those issues in subsequent publications.  

The rest of the paper is organized as follows. First, in Section \ref{sec:preliminaries}, we introduce some basics about  Heisenberg groups and the related Brownian motion on $\Heis$.  For the convenience of readers, some preliminary material on the projective approach to Fourier analysis on $\Heis$ developed in \cite{BCD} is also be presented in this section. Section \ref{construction of noise} is then devoted to the construction of our noises $\dot\bW_{\al}$ on $\Heis$, together with a study of some elementary properties of this family of random fields. Finally, in Section \ref{existence and uniqueness}, we prove the existence and uniqueness of the solution to equation \eqref{eq:pam}. We also demonstrate a exponential upper and lower bound for the second moment of the solution.

\begin{notation}\label{notation-asymp} 
Throughout the paper, we sometimes write $a\lesssim b$ when there exists an unspecified constant $C$ such that $a\le Cb$. In the same way we write $a\asymp b$ if there exist $C_1$, $C_2$ such that $C_1 b\le a\le C_2 b$. 

\end{notation}

\section{Preliminaries}\label{sec:preliminaries}

In this section we first recall some basic facts about Brownian motions on the $(2n+1)$-dimensional isotropic Heisenberg group $\Heis$, that are associated to its sub-Riemannian structure. We also include a short introduction about Fourier analysis on $\Heis$.  
\subsection{The Heisenberg group}\label{sec:heisenberg-group}

The Heisenberg group $\Heis$ is one of the simplest example of manifold with a sub-Riemannian structure. As mentioned in the introduction, it can be identified with $\R^{2n}\times \R$, equipped with the group multiplication: 
\[
(x,y,z)\star (x',y',z'):=\left(x+x', y+y', z+z'+2\omega((x,y),(x',y')) \right),\quad (x,y), (x',y')\in \R^{2n}
\]
where $\omega:\R^{2n}\times\R^{2n}\to\R$, $\omega((x,y),(x',y'))=\sum_{i=1}^n x_i'y_i-x_iy_i'$ is the standard symplectic form on $\R^{2n}$.
The identity in $\Heis$ is $e=(0,0,0)$ and the inverse is given by $(x,y,z)^{-1}=(-x,-y,-z)$. Its Lie algebra $\mathfrak{h}$ can be identified with $\R^{2n+1}$ with the Lie bracket given by
\begin{equation*}
[(a,b, c), (a', b', c')]=\left(0, 0, 2 \omega((a,b), (a',b'))\right), \quad (a,b), (a',b')\in \R^{2n}.
\end{equation*}
Clearly $\mathfrak{h}$ can be identified with the tangent space $T_e(\Heis)$. A basis of left invariant vector fields at $p=(x,y,z)$ is given as
\begin{equation}\label{eq-basis}
X_i(p)=\partial_{x_i}+2y_i\partial_{z}, \quad Y_{i}(p)=\partial_{y_i}-2x_i\partial_{z}, \quad Z(p)=\partial_{z}, \quad i=1,\dots, n.
\end{equation}

In the basis \eqref{eq-basis} we then distinguish a horizontal bundle $\mathcal{D}_p:=\mathrm{Span}\{X_i(p), Y_i(p), i=1,\dots, n\}$, for all $p\in \Heis$. We recall that $\mathcal{D}$ satisfies the basic and fundamental condition $[X_i,Y_i]=Z$ for all $1\le i\le n$.
It is often more geometrically meaningful to describe the Heisenberg group as a sub-Riemannian manifold by equipping the horizontal bundle $\mathcal{D}_p$, $p\in \Heis$ with an inner product such that $\{X_i(p), Y_i(p), i=1,\dots, n\}$ is an orthonormal frame at $p$. The associated  horizontal sub-Laplacian is then given by
\begin{equation}\label{eq-Laplacian}
\Delta=\sum_{i=1}^nX_i^2+Y_i^2.
\end{equation}
Also notice that $\mu$ will designate the Haar measure on $\Heis$, which is nothing else but  the Lebesgue measure on $\mathbb R^{2n+1}$. 

In the remainder of the article we will investigate the invariance of our noises under some natural families of transformations on $\Heis$. More specifically, the 3 families we will consider are the following:
\begin{enumerate}[wide, labelwidth=!, labelindent=0pt, label=\text{(\roman*)}]
\setlength\itemsep{.00in}

\item 
The dilations $\Heis\to \Heis$, which are defined as a family $\{\delta_\lambda;\lambda>0\}$ given by
\begin{equation}\label{eq-dil-delta}
\delta_\lambda(x,y,z)=(\lambda x, \lambda y, \lambda^2 z),\quad\textrm{for any } (x,y,z)\in\Heis.
\end{equation}

\item
The horizontal rotations $R_\theta$, defined by  
\begin{equation}\label{eq-rotation}
R_\theta(x,y,z)=((\cos\theta) x+(\sin\theta) y, -(\sin\theta) x+ (\cos\theta)y,z).
\end{equation}
Here for any $x\in \R^n$, the vector $(\cos\theta) x$ is given by $((\cos\theta) x_1, \dots (\cos\theta) x_n)$.\\

\item
Eventually, for any $x\in\Heis$, we denote by $L_x:\Heis\to\Heis$ the left translation
\begin{equation}\label{eq-left-transl}
L_xy:=xy,\quad\text{for all} \quad y\in \Heis.
\end{equation}
\end{enumerate}

\subsection{Brownian motion on $\Heis$}\label{intro BM on Heis}
In the sequel we will write $\{\mathcal{B}_t\}_{t\ge0}$ for the Brownian motion on $\Heis$.  It is common to define $\mathcal{B}$ as a Markov process with generator $\frac12\Delta$. Given a standard Brownian motion $(B, \beta)$ in $\R^{2n}$, $\mathcal{B}$ can also be realized as the solution of the following SDE interpreted  in the Stratonovich sense:
\begin{equation}\label{eq-BM-sde}
d\BH_t=\sum_{i=1}^nX_i(\BH_t)\circ dB^i_t+Y_i(\BH_t)\circ d\beta^i_t,
\end{equation}
with initial condition $\BH_0=e$. Equation \eqref{eq-BM-sde} can be solved explicitly and for $t\ge0$ we obtain 
\begin{equation}\label{eq-BM}
\BH_t=\left(B_t, \beta_t,A_t\right) ,
\quad\text{where}\quad
A_t=2 \sum_{i=1}^n\int_0^t{B^i_s}d\beta^i_s-\beta^i_sdB_s^{i} .
\end{equation}
Notice that the process $A_t$ in \eqref{eq-BM} is usually referred to as the L\'evy area of $(B,\beta)$.

The density for the distribution of $\BH_t$ at time $t>0$ is given by the heat kernel $p_t$. When issued from the identity $e$ and working with the Haar measure $\mu$ on $\Heis$, a result by Gaveau~\cite{Gaveau} based on the L\'evy area formula for the characteristic function of $A_t$ yields the following semi-explicit form for $p_t$: for any $q=(x,y,z)\in \Heis$ we have
\begin{equation}\label{eq-Heis-kernel}
p_t(q)=\frac{1}{(2\pi t)^{n+1}}\int_\R e^{i\frac{\lambda z}{t}}\left(\frac{2\lambda}{\sinh{2\lambda}}\right)^{n}\exp\left(-\frac{\lambda}{t} |(x,y)|^2\coth (2\lambda) \right)d\lambda,
\end{equation}
where $|(x,y)|^2= x_1^2+\cdots +x_n^2+y_1^2+\cdots +y_n^2$.
As a consequence of our presentation \eqref{eq-Heis-kernel}, one can derive invariances of $p_t$ with respect to the transformations on $\Heis$ introduced in Section~\ref{sec:heisenberg-group}. We label this result here for further use.

\begin{lemma}
Let $p_t$ be the heat kernel defined by \eqref{eq-Heis-kernel}. Then the following invariance properties hold true.
\begin{enumerate}[wide, labelwidth=!, labelindent=0pt, label=\text{(\roman*)}]
\setlength\itemsep{.00in}

\item 
For the dilations $\delta_\lambda$ in \eqref{eq-dil-delta}, we have
\begin{equation}\label{eq-kernel-dil}
p_t\circ \delta_\lambda=\frac{1}{\lambda^{2(n+1)}}p_{\frac{t}{\lambda^2}}.
\end{equation}
\item
For the rotations $R_\theta$ in \eqref{eq-rotation}, it holds that
\begin{equation}\label{eq-kernel-rot}
p_t\circ R_\theta=p_t.
\end{equation}
\item
For the translations $L_x$ introduced in \eqref{eq-left-transl}, we get
\begin{equation}\label{eq-kernel-transl}
p_t\circ L_{x}=p_t.
\end{equation}
\end{enumerate}
\end{lemma}

\begin{proof}
Write $p_{t/\lambda^2}$ according to the semi-explicit formula \eqref{eq-Heis-kernel}. Then it is readily checked that 
\[
p_{\frac{t}{\lambda^2}}(x,y,z)=\lambda^{2(n+1)}p_t(\lambda x, \lambda y, \lambda^2 z),
\]
which proves our claim \eqref{eq-kernel-dil}. Relations \eqref{eq-kernel-rot} and \eqref{eq-kernel-transl} are obtained similarly.
\end{proof}
In order to get proper bounds on the kernel defined by \eqref{eq-Heis-kernel}, we first introduce a distance on $\Heis$ that accommodates the sub-Riemannian structure mentioned in Section~\ref{sec:heisenberg-group}. To this end, we call a path $\gamma:[0,1]\to \Heis$ horizontal if it is absolutely continuous and $\dot{\gamma}(t)\in \ch_{\gamma(t)}$ for all $t\in [0,1]$. The Carnot-Carath\'eodory distance on $\Heis$ is then defined as
\begin{equation}\label{eq-cc-dist}
d_{cc}(p_1,p_2):=\inf\bigg\{\int_0^1 |\dot{\gamma}(t)|_{\ch}dt \,
 ; \, 
 \gamma:[0,1]\to \Heis \ \mbox{is horizontal}, \ \gamma(0)=p_1,\gamma(1)=p_2\bigg\}.
\end{equation}
In fact the CC-distance is equivalent to the so-called homogeneous distance on $\Heis$. Specifically, there exist $C_1,C_2>0$ such that 
\begin{equation}\label{eq-cc-dist-homo}
C_1(\sqrt{|(x,y)|^2}+|z|^{\frac12})\le d_{cc}(e, q)\le C_2(\sqrt{|(x,y)|^2}+|z|^{\frac12})).
\end{equation}
A striking property of the CC-distance is that the Hausdorff dimension of the metric space $(\Heis,  d_{cc})$ is equal to $2n+2$ and thus strictly greater  than the topological dimension of $\Heis$ which is $2n+1$.

With those definitions in hand, it is known that the heat kernel $p_t$ defined by \eqref{eq-Heis-kernel} satisfies some Gaussian type upper and lower bounds (see references \cite{BG, Li07}). Namely there exist $c_1,c_2,c_3,c_4>0$ such that for every $t>0$ and $q=(x,y,z)$ we have
\begin{equation}\label{eq-heat-kernel-bds}
\frac{c_{1}}{t^{n+1}}\exp\biggl(-\frac{c_{2}}{t}  \, d_{cc}(e,q)^{2}\biggr) 
 \le p_{t}(q)
 \leq \frac{c_{3}}{t^{n+1}}\exp\biggl(-\frac{c_{4}}{t}  \, d_{cc}(e,q)^{2}\biggr)  .
\end{equation}

\subsection{Fourier analysis and Sobolev spaces}\label{sec-fpower-lap}

The sub-Laplacian $\Delta$ defined by \eqref{eq-Laplacian}  is essentially self-adjoint on the space of smooth and compactly supported functions $\mathcal{C}_0^\infty(\mathbb{R}^{2n+1})$, see \cite{MR3587668}. Therefore it admits a unique self-adjoint extension, which we still denote by $\Delta$, that is defined on a dense subspace of $L^2(\mathbb{R}^{2n+1},\mu)$, where $\mu$ designates as before the Lebesgue measure on $\R^{2n+1}$.

The above considerations lead to a spectral representation of $\Delta$ (for instance see \cite{RS}). Namely from the spectral theorem, there exist a measure space $(\Lambda, \nu)$, a unitary map $U: L^2(\Heis,\mu) \to L^2 (\Lambda, \nu) $ and  a non negative real valued measurable function $m$ on $\Lambda$ such that
\begin{equation}\label{eq-spectral}
U \Delta U^{-1} g (\lambda)=-m(\lambda) g(\lambda), 
\end{equation}
for all $\lambda \in \Lambda$, and $g\in L^2 (\Lambda, \nu)$ such that $U^{-1}g \in \dom(\Delta)$. Note that given $g \in  L^2 (\Lambda,\nu)$, we have that $U^{-1}g$ belongs to $\dom(\Delta)$ if only if 
\begin{equation}\label{eq-L2-cond}
\int_{\Lambda} m(\lambda)^2 g(\lambda)^2 \nu\left(d\lambda\right) <+\infty.
\end{equation}
 The spectral decomposition \eqref{eq-spectral} allows for the definition of fractional powers of $\Delta$. That is for $\alpha \in \mathbb{R}$, we define the fractional sub-Laplacian $(-\Delta)^{-\alpha}$ as an unbounded and densely defined operator on $L^2(\Heis,\mu)$ defined by
\[
U (-\Delta)^{-\alpha} U^{-1} g (\lambda)=m(\lambda)^{-\alpha} g(\lambda).
\]
The domain of $(-\Delta)^{-\alpha}$ is a Sobolev space denoted by $\mathcal{W}^{-\alpha,2}$ and defined as 
\begin{equation}\label{eq-soblev-sp}
\mathcal{W}^{-\alpha,2}=\bigg\{U^{-1}g:\Heis\to \R; \int_\Lambda m(\lambda)^{-2\alpha}g(\lambda)^2\nu(d\lambda)<\infty\bigg\}.
\end{equation}

The construction \eqref{eq-soblev-sp} for the Sobolev space $\mathcal{W}^{-\alpha,2}$ is quite abstract. A more concrete and explicit version of $\mathcal{W}^{-\alpha,2}$ can be described through Fourier transforms on the Heisenberg group. Below we will follow~\cite{BCD} for this alternative construction. Indeed, it is well-known that all irreducible representations of the Heisenberg group $\Heis$ are unitary equivalent to the Schr\"odinger representations $(U^\lambda)_{\lambda \in \mathbb{R}}$, that is the family of group morphisms $q=(x,y,z) \in \Heis \to U_q^\lambda$ between $\Heis$ and the unitary group of $L^2(\R^{n})$ defined by
\begin{equation}\label{eq-uni-repre}
U_q^\lambda u (\xi) = e^{-i\lambda\left( z+2x\cdot(\xi- y) \right) } u \left( \xi -2y\right), \quad \xi \in \mathbb R^n
\end{equation}
The group Fourier transform of a function $f \in L^1(\Heis)$ is defined for each $\lambda \in \R\setminus\{0\}=:\R^*$ as the operator valued function on $L^2(\mathbb R)$ given by
\begin{equation}\label{eq-Fourier-transf}
\mathcal{F}(f) (\lambda)=\int_{\Heis} f(q)U_{q}^\lambda \,  d\mu(q) .
\end{equation}
Clearly $\mathcal{F}(f) $ takes values in the space of bounded operators on $L^2(\R^n)$.
If in addition $f \in L^2(\Heis)$, $\mathcal{F}(f) (\lambda)$ is a Hilbert-Schmidt operator and the following Plancherel theorem holds,
\begin{equation}\label{eq:plancherel-operator}
\int_{\Heis} |f(q)|^2 d\mu(q)=\frac{2^{n-1}}{\pi^{n+1}}\int_{-\infty}^{+\infty}\| \mathcal{F}(f) (\lambda) \|_{\textsc{hs}}^2 
\, |\lambda |^n 
\, d\lambda .
\end{equation}

The fact that representations of the Heisenberg group are operator-valued leads to cumbersome considerations for the Fourier transform. 
In order to get a more tractable version of Fourier type computations, a projective point of view is advocated in~\cite{BCD}. Namely for $\lambda \in \mathbb R^*$ and $k \in \mathbb N^n$, we consider the family of Hermite functions
\[
\Phi_k^\lambda (x)= |\lambda|^{n/4} \Phi_k \lp \sqrt{|\lambda|} x\rp, \quad x \in \mathbb R^n,
\]
where $ \Phi_k$ is the normalized Hermite function on $\mathbb R^n$ which is an eigenfunction of the Schr\"odinger operator $H=-\sum_{j=1}^n\frac{d^2}{dx_j^2}+ |x|^2$ with eigenvalue $2|k|+n$. Here for any multi-index integer $k=(k_1,\dots, k_n)\in \N^n$ we define $k!:=k_1!\cdots k_n!$ and $|k|:=k_1+\cdots +k_n$. This family is an orthonormal basis of $L^2(\mathbb R^n)$, therefore we get
\begin{equation}\label{eq:projective-hs-norm}
\| \mathcal{F}(f) (\lambda) \|_{\textsc{hs}}^2=\sum_{k\in \N^n}   \| \mathcal{F}(f) (\lambda)  \Phi_k^\lambda \|_{L^2(\R^n)}^2.
\end{equation}
Identity \eqref{eq:projective-hs-norm} leads us to introduce the following projective definition of the Fourier transform.

\begin{definition}\label{def:projective-fourier}
Let $f$ be an element of $L^2(\Heis)$.
If $(m,\ell,\lambda) \in \mathbb N^n \times  \mathbb N^n  \times \mathbb R^*$ we denote 
\begin{equation}\label{eq-F-trans-decomp}
\hat f (m,\ell,\lambda)= \langle \mathcal{F}(f) (\lambda) \Phi_m^\lambda , \Phi_\ell^\lambda \rangle_{L^2(\R^n)} .
\end{equation}
We will call $\hat f $ the Fourier transform of $f$.
\end{definition} 
A few remarks about Definition~\ref{def:projective-fourier} are in order. First, notice that in \eqref{eq-F-trans-decomp}, the Fourier transform is now complex valued as in $\R^{n}$. Next observe that the family $U$ defined by \eqref{eq-uni-repre} satisfies $U_e^\lambda=\text{Id}$. Hence the projective Fourier transform of the Dirac distribution at the identity $e$ is given by:
\[
\hat {\delta_e} (m,\ell,\lambda)=\langle \mathcal{F}(\delta_e) (\lambda) \Phi_m^\lambda , \Phi_\ell^\lambda \rangle_{L^2(\R^n)}= \langle U_e^\lambda (\Phi_m^\lambda ), \Phi_\ell^\lambda \rangle_{L^2(\R^n)}=\delta_{m,\ell},
\]
where  $\delta_{m,\ell}=1$ if $m =\ell$ and $0$ otherwise. Moreover, with Definition \ref{def:projective-fourier} in hand, the Plancherel formula \eqref{eq:plancherel-operator} takes a more familiar shape:
\begin{align}\label{proj Plancherel}
\int_{\Heis} |f(q)|^2 dq=\frac{2^{n-1}}{\pi^{n+1}}\sum_{m,\ell\in \N^n} 
\int_{-\infty}^{+\infty}  | \hat f (m,\ell,\lambda)|^2  |\lambda |^n d\lambda .
\end{align}
In addition, the Fourier inversion formula for $q=(x,y,z)$ is also obtained by integrating complex valued functions. Namely we get
\[
f(q) = \sum_{m,\ell\in \N^n}\int_{\mathbb{R}} e^{i\lambda z} K_{m,\ell,\lambda} (q) \hat f (m,\ell,\lambda)d\lambda ,
\]
where $K_{m,\ell,\lambda} (q)$  is a kernel that can be expressed in terms of the Hermite functions (see ~\cite{BCD}).

Definition \ref{def:projective-fourier} enables a convenient notion of fractional sub-Laplacians and Sobolev spaces. This construction stems from the fact that $\Delta$ acts nicely on projective Fourier transforms, leading to the following simple formula for a function $f$ in the Schwartz space $\mathcal{S}(\mathbb R^{2n+1})$ of rapidly decreasing functions on $\mathbb R^{2n+1}$:
\begin{equation}\label{eq:Laplace-on-Fourier}
\widehat{\Delta f}(m,\ell,\lambda)= - 4 \, |\lambda| (2|m|+n) \, \widehat{f}(m,\ell,\lambda),
\quad\text{for}\quad
(m,\ell,\lambda) \in \mathbb N^n \times  \mathbb N^n  \times \mathbb R^*.
\end{equation}
If we denote by $(P_t)_{t \ge 0}$ the semigroup generated by $\Delta$, then formula \eqref{eq:Laplace-on-Fourier} shows that  for a function $f$ in the Schwartz space $\mathcal{S}(\mathbb R^{2n+1})$,
\[
\widehat{P_tf}(m,\ell,\lambda)= e^{- 4t \, |\lambda| (2|m|+n)} \, \widehat{f}(m,\ell,\lambda).
\]
It follows that if, as before, $p_t$ denotes the heat kernel issued from the identity $e$, then
\begin{align}\label{HeatKernelFourier}
\widehat{p}_t (m,\ell,\lambda)= e^{- 4t \, |\lambda| (2|m|+n)} \delta_{m,\ell}.
\end{align}
Starting from \eqref{eq:Laplace-on-Fourier} and considering $f \in \mathcal{S}(\mathbb R^{2n+1})$, it is thus natural to define $(\Delta)^{-\alpha}f$ for $\al\in(0,1)$ through its Fourier transform: for $(m,\ell,\lambda) \in \mathbb N^n \times  \mathbb N^n  \times \mathbb R^*$ we have
\begin{equation}\label{Fourier transform Laplace}
\widehat{(-\Delta)^{-\alpha}f}(m,\ell,\lambda)=4^{-\alpha} |\lambda|^{-\alpha} (2|m|+n)^{-\alpha} \widehat{f}(m,\ell,\lambda) \, .
\end{equation}
One can then check that $(-\Delta)^{-\alpha}f $ sits in $\mathcal{S}(\mathbb R^{2n+1})$. Related to our fractional Laplacians, the Sobolev space $\mathcal{W}^{-\alpha,2}$ can then be described as
\begin{equation}\label{eq:frac-sobolev-fourier}
\mathcal{W}^{-\alpha,2}= \left\{ f \in L^2(\Heis), \int_{-\infty}^{+\infty} \left(\sum_{m,\ell\in \N^n} |\lambda|^{-2\alpha} (2|m|+n)^{-2\alpha}  \hat f (m,\ell,\lambda)^2 \right) |\lambda |^n d\lambda <+\infty   \right\}  .
\end{equation}

Let us finish this section by some estimates for the fractional Laplacian $(-\Delta)^{-\alpha}$. Specifically, invoking the Gaussian estimates \eqref{eq-heat-kernel-bds}, it can be shown (as in \cite{BL}) that for $0<\alpha<n+1 $ the operator $(-\Delta)^{-\alpha}$ admits a kernel $G_\alpha (q_1,q_2)$ which is given by:
\begin{align}\label{green function}
G_\alpha(q_1,q_2)=\frac{1}{\Gamma(s)} \int_0^{+\infty} t^{\alpha-1} p_t(q_1,q_2) \, dt, 
\quad\text{for all}\quad
 q_1,q_2 \in \Heis.
\end{align}
Then using the representation \eqref{green function}, one can obtain the following bound for the kernel $G_\alpha$.
\begin{lemma}
Let $0<\alpha<n+1 $ and consider $G_{\al}$ defined by \eqref{green function}. Then there exist constants $c,C>0$ such that for $\mu$ a.e. $ q_1,q_2 \in \Heis$, we have
\begin{equation}\label{eq-kernel-cc}
\frac{c}{d_{cc}(q_1,q_2)^{2(n+1-\alpha)}} \le G_\alpha(q_1,q_2) \le \frac{C}{d_{cc}(q_1,q_2)^{2(n+1-\alpha)}}.
\end{equation}
\end{lemma}

\begin{proof}
We invoke the left-invariant property of both $p_t$ (see \eqref{eq-kernel-transl}) and $d_{cc}$. Then we integrate the heat kernel bounds \eqref{eq-heat-kernel-bds} into \eqref{green function}. We obtain that
\begin{align*}
G_\alpha(q_1,q_2)&\le \frac{1}{\Gamma(s)} \int_0^{+\infty} t^{\alpha-1} p_t(q_1^{-1}q_2) \, dt \\
&\le C'\int_0^{+\infty} t^{\alpha-n-2}  \exp\biggl(-\frac{c_{4}}{t}  \, d_{cc}(e,q_1^{-1}q_2)^{2}\biggr) dt\\
&= \frac{C}{d_{cc}(q_1,q_2)^{2(n+1-\alpha)}}
\end{align*}
for some constants $C,C'>0$. We can show the lower bound in a similar way. Hence we obtain the conclusion.
\end{proof}

\section{Description of the fractional  noise $\mathbf{\dot{W}}$}\label{construction of noise}
A good understanding of \eqref{eq:pam} relies on a proper description of the driving noise $\mathbf{\dot{W}}$. We carry out this task here, including cases where $\mathbf{\dot{W}}$ is distribution-valued as well as  function-valued. Our spatial covariance functions will all be based on fractional powers of $\Delta$ as introduced in Section \ref{sec-fpower-lap}. More specifically we consider a noise on $\R_+\times \Heis$ which is white in time and colored in space, similarly to what is done in $\R^d$ (see for instance \cite{Da} and \cite{PZ}). For sake of conciseness, the case of a  colored noise in time (cf. e.g \cite{HHNT}) is postponed to a subsequent publication. In this context, given a non-negative definite function 
 $\Lambda: \Heis\to\R_+$, the noise $\bW$ will be thought of as a centered Gaussian family $\{\bW(\varphi); \varphi\in \cd\}$ for a certain subset $\cd$ of functions on $\R_+\times \Heis$. The covariance of $\bW$ is then given by
\begin{equation}\label{eq-cov}
\bE\left( \bW(\varphi)\bW(\psi)\right)=\int_{\R_+}\int_{\Heissq}\varphi(t,q_1)\psi(t,q_2)\Lambda(q_2^{-1}q_1)d\mu (q_1)d\mu (q_2)dt,
\end{equation}
where we recall that $\mu$ stands for the Haar measure on $\Heis$.
As mentioned above, we will now introduce two important types of covariance functions, generated by powers of the sub-Laplacian $\Delta$ on $\Heis$.

\subsection{Negative powers of the sub-Laplacian with $\mathbf{0<\alpha<\frac{n+1}2}$}\label{sec:s-less-than-1}

In this section we consider a distribution valued noise $\mathbf{\dot{W}}_{\al}$ whose covariance function is generated by $(-\Delta)^{-\alpha}$, where we recall that $\Delta$ is defined by \eqref{eq-Laplacian}. Namely we wish the  covariance function of $\mathbf{\dot{W}}_\alpha$ to be given by 
\begin{align}\label{eq-cov-0<s<1}
\bE\left( \bW_{\!\al}(\varphi)\bW_{\!\al}(\psi)\right)&=\int_{\R_+}\int_{\Heis} (-\Delta)^{-\alpha} \varphi (t,q_1) (-\Delta)^{-\alpha} \psi (t,q_1) \, d\mu(q_1) dt .
\end{align}
Notice that if such a covariance function exists, it is obviously a positive definite function according to expression \eqref{eq-cov-0<s<1}. We now state a proposition ensuring the existence of $\mathbf{\dot{W}}$ as a Gaussian family. 
\begin{proposition}\label{prop-field-W}
Let $\alpha$ be a regularity parameter in  $\left(0,\frac{n+1}2\right)$. We also introduce a Hilbert space $\mathcal{H}$ as 
\begin{align}\mathcal{H}= L^2(\R_+;\mathcal{W}^{-\alpha,2}),\label{def mathscrH}
\end{align} 
where $\mathcal{W}^{-\alpha,2}$ is the space defined by \eqref{eq:frac-sobolev-fourier}. Then the covariance function~\eqref{eq-cov-0<s<1} defines an isonormal centered family $\{\bW_{\!\al}(\varphi);$ $\varphi\in\mathcal{H} \}$. Moreover, 
\begin{enumerate}[wide, labelwidth=!, labelindent=0pt, label=\textnormal{(\arabic*)}]
\setlength\itemsep{.05in}

\item 
The covariance $\bE\left( \bW_{\!\al}(\varphi)\bW_{\!\al}(\psi)\right)$ admits the representation \eqref{eq-cov}, with $\Lambda=G_{2\alpha}$ that is given by \eqref{green function}.
\item 
If $\varphi,\psi$ are two non negative functions in $\mathcal{H}$ we have 
\[
\bE\left( \bW_{\!\al}(\varphi)\bW_{\!\al}(\psi)\right)\asymp \int_{\R_+}\int_{\Heissq} \frac{\varphi (t,q_1) \psi (t,q_2)}{d_{cc}(q_1,q_2)^{2n+2-4\alpha}} \, dt\,d\mu(q_1) d\mu(q_2).
\]
where we recall that the symbol $\asymp$ is introduced in Notation \ref{notation-asymp}.
\end{enumerate}
\end{proposition}
\begin{proof}
First it is obvious from \eqref{eq-cov-0<s<1} that $\bE( \bW_{\!\al}(\varphi)\bW_{\!\al}(\psi))$ is well defined for $\varphi, \psi\in \mathcal{W}^{-\alpha,2}$. Next in order to prove our first claim, we start from expression \eqref{eq-cov-0<s<1} and use the fact that $(-\Delta)^{-\alpha}$ is self-adjoint in order to write 
\begin{align}\label{covariance by Laplacian}
\bE\left( \bW_{\!\al}(\varphi)\bW_{\!\al}(\psi)\right)=\int_{\R_+} \int_{\Heis} \varphi (t,q_1)(-\Delta)^{-2\alpha} \psi (t,q_1)d\mu(q_1) dt. 
\end{align}
Since $(-\Delta)^{-2\alpha}$ admits the kernel $G_{2\alpha}$ given by \eqref{green function}, we can recast the above equality as
\begin{equation}\label{eq-d}
\bE\left( \bW_{\!\al}(\varphi)\bW_{\!\al}(\psi)\right)=\int_{\R_+} \int_{\Heissq}G_{2\alpha}(q_1,q_2) \varphi(t,q_1) \psi (t,q_2) d\mu(q_1) d\mu(q_2)dt,
\end{equation}
which proves our first claim. Eventually the second claim is easily shown by combining~\eqref{eq-d} and the estimate \eqref{eq-kernel-cc}.
\end{proof}

\begin{remark}
For $\alpha \in  \left(0,\frac{n+1}2\right]$ an alternative way to construct and study our noise $\bW_{\!\al}$ is to use the theory of tempered distribution  in the Heisenberg  groups developed in \cite{BCD}.  More precisely, it turns out that the space of tempered distributions $\mathcal{S}'(\Heis)$ on the Heisenberg group  is the same as the space of tempered distributions in $\mathbb{R}^{2n+1}$. We can then define fractional powers of the sub-Laplacian on tempered distributions by the duality formula
\[
\left\langle (-\Delta)^{-\alpha}  \Phi , f \right\rangle= \left\langle   \Phi ,(-\Delta)^{-\alpha} f \right\rangle, \quad \Phi \in \mathcal{S}'(\Heis), f \in \mathcal{S}(\Heis).
\]
and it follows from the estimates in  \cite{BCD} that  $(-\Delta)^{-\alpha}: \mathcal{S}'(\Heis) \to \mathcal{S}'(\Heis)$. In particular, one can then define the random tempered distribution
\[
\bW_{\!\al}= (-\Delta)^{-\alpha} \bW \, , 
\]
where $\bW$ is a white noise (realized as a tempered distribution) on $[0,+\infty)\times \Heis$. This remark will be further developed in a subsequent work, where weighted Besov spaces of tempered distributions supporting the noise $\bW_{\!\al}$ will be explicitly constructed. In the present work, we will only rely on the Hilbert space type construction provided in Proposition \ref{prop-field-W}.
 \end{remark}

Next we state some basic properties of the fields $\bW_{\!\al}$, namely invariance by dilation and rotation. To this aim, we label the invariance of our potential $G_\alpha$ by dilation for further use. 
\begin{lemma}
Let $G_\alpha$ the kernel for the operator $(-\Delta)^{-\alpha}$ with $0<\alpha<n+1$. For the dilations $\delta_\lambda$ in \eqref{eq-dil-delta} we have the following scaling property:
\begin{equation}\label{eq-dil-kernels}
 G_\alpha\circ\delta_\lambda=\frac{1}{\lambda^{2(n+1-\alpha)}}G_\alpha.
\end{equation}
\end{lemma}
\begin{proof}
Write $G_\alpha\circ \delta_\lambda$ thanks to relation \eqref{green function}. Then apply the elementary change of variable $t/\lambda^2=u$ therein and invoke \eqref{eq-kernel-dil}. This proves \eqref{eq-dil-kernels} thanks to elementary considerations. 
\end{proof}

We are now ready to state the dilation invariance property for $\bW_\alpha$ when $0<\alpha<\frac{n+1}2$.

\begin{proposition}\label{prop-dilation}
Let $\bW_{\!\al}$  be the field described in Proposition \ref{prop-field-W}, with $0<\al<\frac{n+1}2$.  For any $\lambda>0$, we have that
\begin{equation}\label{eq-dilation}
\lbrace \bW_{\!\al}(\varphi\circ \delta_\lambda), \varphi\in L^2(\R_+; \cw^{-\alpha,2})\rbrace\stackrel{\mathcal{D}}{=}\lbrace\lambda^{-(n+1+2\alpha)} \bW_{\!\al}(\varphi), \varphi\in L^2(\R_+;\cw^{-\alpha,2})\rbrace.
\end{equation}
\end{proposition}
\begin{proof}
Both sides of \eqref{eq-dilation} define a centered Gaussian field. We shall prove that the covariance functions of those two fields coincide. Let us start with the left hand side of \eqref{eq-dilation}, and consider two functions $\varphi, \psi$ in $L^2(\R_+;\cw^{-\alpha,2})$. Owing to \eqref{eq-d} we have
\begin{align*}
&\bE\left( \bW_{\!\al}(\varphi\circ \delta_\lambda)\bW_{\!\al}(\psi \circ \delta_\lambda)\right)=\int_{ \R_+\times\Heissq}\varphi(t,\delta_\lambda p)\psi(t,\delta_\lambda q)G_{2\alpha}(p,q)dt\, d\mu (p)d\mu (q)
\end{align*}
Next we set $p:=\delta_\lambda p, q:=\delta_\lambda q$ in the integral above. This yields
\begin{multline}\label{eq-dilation-b}
\bE\left( \bW_{\!\al}(\varphi\circ \delta_\lambda)\bW_{\!\al}(\psi \circ \delta_\lambda)\right) \\
=\int_{\R_+\times \Heissq}\varphi(t,  p)\psi(t, q)G_{2\alpha}(\delta_\lambda^{-1}p,\delta_\lambda^{-1}q) dt\,d\mu (\delta_\lambda^{-1} p) d\mu (\delta_\lambda^{-1}q) .
\end{multline}
Plugging relation \eqref{eq-dil-kernels} into \eqref{eq-dilation-b}, and resorting to the  property $d\mu(\delta_{\lambda}^{-1}q)=\lambda^{-Q}d\mu(q)$ with $Q=2(n+1)$, we get
\begin{align*}
\bE\left( \bW_{\!\al}(\varphi\circ \delta_\lambda)\bW_{\!\al}(\psi \circ \delta_\lambda)\right)=\lambda^{-2(n+1+2\alpha)}\int_{\R_+\times \Heissq}\varphi(t, p)\psi(t, q)G_{2\alpha}(p,q)dt\, d\mu (p)d\mu (q).
\end{align*}
This is obviously the covariance function for the right hand side of \eqref{eq-dilation}, which ends the proof.
\end{proof}

\begin{remark}
Identity \eqref{eq-dilation} explains why we separate between $\alpha<\frac{n+1}{2}$ and $\alpha>\frac{n+1}{2}$ for a distributional vs function-valued noise $\bW_\alpha$. In order to see this, for a parameter $\lambda>0$, a given $x\in\Heis$ and a smooth function $\varphi$, set 
\begin{align*}
\mathcal{S}_{x,\lambda}\varphi(p)=\frac{1}{\lambda^Q}\varphi(\delta_{\lambda^{-1}}(x^{-1}p)),\quad p\in\Heis,
\end{align*}
where we recall that $Q=2(n+1)$ is the homogeneous (or effective) dimension. Then a way to characterize the fact that a distribution $f$ is in $\mathcal{C}^\beta$ for some $\beta<0$ on $\Heis$ is to have 
\begin{align}\label{char beta-holder}
\int_{\Heis}\mathcal{S}_{x,\lambda}\varphi(p)f(p) \, d\mu(p)\leq c_\varphi \lambda^\beta,
\end{align}
for all test functions $\varphi$. Note that the criterion \eqref{char beta-holder} is also valid for negative values of $\beta$. At a heuristic level, one can implement the criterion \eqref{char beta-holder} to our noise $\bW_\alpha$ by checking that 
\begin{align}
\bE\left[\left(\bW_\alpha(\mathcal{S}_{x,\lambda}\varphi)\right)^2\right]\leq c_\varphi\lambda^{2\beta}.\label{test noise holder}
\end{align}
Now \eqref{eq-dilation} easily yields
\begin{eqnarray*}
\bE\left[\left(\bW_\alpha(\mathcal{S}_{x,\lambda}\varphi)\right)^2\right]
&=&\frac{1}{\lambda^{2Q}}\bE\left[\bW_\alpha\left(\varphi(x^{-1}p)\circ\delta_{\lambda^{-1}}\right)\right]\\
&=&c_\varphi\lambda^{2(n+1+2\alpha)-2Q}=\lambda^{4\alpha-2(n+1)}.
\end{eqnarray*}
We thus get an exponent $\beta=2\alpha-(n+1)$ in \eqref{char beta-holder}. In particular, the separation between $\beta<0$ (distribution-valued noise) and $\beta>0$ (function-valued noise) occurs at $\alpha=\frac{n+1}{2}$. This explains our restriction on $\alpha$ for this section as well as Section \ref{function val alpha}.
\end{remark}

Next we consider invariances with respect to the horizontal rotations $R_\theta$ defined by \eqref{eq-rotation}. Our result is summarized in the proposition below.
\begin{proposition}\label{prop-rotation}
Let $\alpha\in\left(0,\frac{n+1}2\right)$ and let $\bW_\alpha$ be the same random field as in Proposition~\ref{prop-dilation}. For $\theta\in[0,2\pi)$, consider the rotation $R_\theta$ defined by \eqref{eq-rotation}. Then the following identity in distribution holds true
\begin{equation}\label{eq-rotation-d}
\lbrace \bW_{\!\al}(\varphi\circ R_\theta); \varphi\in L^2(\R_+; \cw^{-\alpha,2})\rbrace\stackrel{\mathcal{D}}{=}\lbrace \bW_{\!\al}(\varphi); \varphi\in L^2(\R_+; \cw^{-\alpha,2})\rbrace.
\end{equation}
\end{proposition}
\begin{proof}
The heat kernel $p_t$ is invariant by rotation, as assessed by \eqref{eq-kernel-rot}. This immediately yields the invariance of $G_\alpha$ by rotation, similarly to what we did for \eqref{eq-dil-kernels}. Eventually the invariance of $\bW_\alpha$ is obtained along the same lines as for Proposition \ref{prop-dilation}.
\end{proof}

We finish this section with invariance for the left translations $L_x$ introduced in Section~\ref{sec:heisenberg-group}.   The homogeneity property is given as below.
\begin{proposition}\label{prop-translation}
For $\alpha\in(0,(n+1)/2)$, let $\bW_\alpha$  be the field as described in Proposition~\ref{prop-dilation}.  For any $x\in \Heis$, we have that
\begin{equation}\label{eq-translation}
\lbrace \bW_{\!\al}(\varphi\circ L_x), \varphi\in L^2(\R_+; \cw^{-\alpha,2})\rbrace\stackrel{\mathcal{D}}{=}\lbrace \bW_{\!\al}(\varphi), \varphi\in L^2(\R_+;\cw^{-\alpha,2})\rbrace \, .
\end{equation}
\end{proposition}
\begin{proof}
Similarly to what we did for Proposition~\ref{prop-dilation}, in order
to prove \eqref{eq-translation} we just need to show that the covariance functions of the two fields agree. Consider the covariance of the noise on the left hand side. For any $\varphi, \psi$ in $L^2(\R_+;\cw^{-\alpha,2})$ and $x\in \Heis$, by \eqref{eq-d} we have 
\begin{align*}
&\bE\left( \bW_{\!\al}(\varphi\circ L_x)\bW_{\!\al}(\psi \circ L_x)\right)=\int_{ \R_+\times\Heissq}\varphi(t, xp)\psi(t,xq)G_{2\alpha}(p,q)dt\, d\mu (p)d\mu (q)
\end{align*}
Due to the left translation invariance of the Haar measure $\mu$, we obtain that
\begin{align}\label{eq-trans-a}
&\bE\left( \bW_{\!\al}(\varphi\circ L_x)\bW_{\!\al}(\psi \circ L_x)\right)=\int_{ \R_+\times\Heissq}\varphi(t, p)\psi(t,q)G_{2\alpha}(x^{-1}p,x^{-1}q)dt\, d\mu (p)d\mu (q)
\end{align}
Then by the translation invariance of the green function, namely
\begin{equation}\label{eq-G-trans}
G_{2\alpha}(x^{-1}p,x^{-1}q)=G_{2\alpha}(p,q),
\end{equation}
we obtain the conclusion. Notice that \eqref{eq-G-trans} is an easy consequence of \eqref{eq-left-transl} and \eqref{green function}.
\end{proof}

\subsection{Negative powers of the sub-Laplacian with  $\mathbf{n/2+1/2<\alpha<n/2+1}$}\label{sec:s-larger-than-1}\label{function val alpha}

In the regime $0<\alpha<\frac{n+1}2$, the noise $\dot{\bW}$ was defined as a distribution. We will now describe a range of parameters for which the fractional Gaussian field can be defined pointwisely as the analogue of a fractional Brownian motion.

We start by giving an integrability result for the kernel $G_{\al}$, which will be crucial in order to define our Gaussian field.
\begin{proposition}\label{prop-integrability-G}
Let $\frac{n+1}2 < \alpha < \frac{n}2+1$ and consider the kernel $G_\alpha$ defined by \eqref{green function}. Then there exists a constant $C>0$ such that for every $x,y \in \Heis$
\begin{equation}\label{a4}
\int_{\Heis} (G_\alpha (x,q)-G_\alpha(y,q))^2  d\mu(q) \le Cd_{cc}(x,y)^{4\alpha-2(n+1)} ,
\end{equation}
where we recall that the distance $d_{cc}$ is given by \eqref{eq-cc-dist}.
\end{proposition}

Before proving Proposition \ref{prop-integrability-G}, let us state and prove two technical lemmas. The first one quantifies the ultracontractivity of the heat semigroup $P_{t}$.
\begin{lemma}
	\label{contPt1}
Let $P_t$ be the heat semigroup on $\Heis$, whose kernel is the function $p_t$ as given by \eqref{eq-Heis-kernel}. There exists a constant $C>0$ such that for every $f \in L^2(\Heis,\mu)$, $t  >0$ and $g \in \Heis$,
\begin{equation}\label{eq-pt-bd}
| P_t f(g) | \le  \frac{C}{t^{\frac{n+1}2}} \| f \|_{L^2(\Heis,\mu)}.
\end{equation}
\end{lemma}

\begin{proof}
Starting from the expression 
$
P_tf(g)=\int_\Heis p_t(g,q)f(q)d\mu(q),
$
we invoke Cauchy-Schwarz inequality and Chapman-Kolmogorov identity in order to get
\begin{align*}
| P_t f(g) | &\le  \|f\|_{L^2(\Heis,\mu)} \left( \int_\Heis p^2_{t}(q,g)d\mu(q)\right)^{1/2}\\
&= \|f\|_{L^2(\Heis,\mu)} \left( p_{2t}(g,g)\right)^{1/2}.
\end{align*}
Hence owing to \eqref{eq-heat-kernel-bds} and due to the left translation invariance property of the heat kernel $p_t$ we have that
\[
| P_t f(g) | \le   \frac{\sqrt{c_3}}{t^{\frac{n+1}2}} \|f\|_{L^2(\Heis,\mu)} ,
\]
where $c_3$ is the constant in \eqref{eq-heat-kernel-bds}. This completes the proof of~\eqref{eq-pt-bd}.
\end{proof}

Our second lemma gives more information about the regularity of $P_{t}f$ for a square integrable function.
\begin{lemma}
	\label{contPt2}
Let the notation of Lemma \ref{contPt1} prevail. There exists a constant $C>0$ such that for every  $f \in L^2(\Heis,\mu)$, $t > 0$ and  $x,y \in \Heis$,
\begin{equation}\label{eq-pt-cont}
| P_t f(x) -P_tf(y) | \le C \frac{d_{cc}(x,y)}{t^{\frac{n}2+1}} \| f \|_{L^2(\Heis,\mu)}.
\end{equation}
\end{lemma}
\begin{proof}
From the reverse Poincar\'e inequality in \cite{BB}, it is known that   there exists a constant $C>0$ such that for every $g \in L^\infty(\Heis,\mu)$, $t > 0$ and $x,y \in \Heis$,
\begin{align}\label{p_tg(x)-p_tg(y)}
| P_t g(x) -P_tg(y) | \le C \frac{d_{cc}(x,y)}{t^{\frac{1}{2}}} \| g \|_{L^\infty(\Heis,\mu)}.
\end{align}
We now invoke the fact that $P_{2t}f=P_tg$ with $g=P_tf$. then we apply successively inequality~\eqref{p_tg(x)-p_tg(y)} and Lemma \ref{contPt1}. This yields
\begin{equation*}
| P_{2t} f(x) -P_{2t}f(y) |
 \le C\frac{d_{cc}(x,y)}{t^{\frac{1}{2}}}\|P_tf\|_{L^\infty(\Heis,\mu)} 
\leq C \frac{d_{cc}(x,y)}{t^{\frac{n}2+1}} \| f \|_{L^2(\Heis,\mu)},
\end{equation*}
which concludes the proof.
%{\color{red} Proposition 6.1.1 in Bonnefont's thesis (page 131) seems to directly imply the above conclusion.}
\end{proof}

We now focus our attention to the proof of Proposition \ref{prop-integrability-G}, for which some of the arguments are inspired by \cite{BL}.
\begin{proof}[Proof of Proposition \ref{prop-integrability-G}]
Consider $x,y\in \Heis$ and $\alpha>0$. Our first aim is to upper bound the quantity
\begin{equation}\label{eq-Q_s}
Q_\alpha f(x,y):=(-\Delta)^{-\alpha}f(x)-(-\Delta)^{-\alpha}f(y),
\end{equation}
for a generic function $f\in L^2(\Heis)$. Note that according to \eqref{green function} we have
\[
Q_\alpha f(x,y)=\frac1{\Gamma(\alpha)}\int_0^\infty t^{\alpha-1}(P_tf(x)-P_tf(y))dt.
\]
We now introduce a parameter $\delta>0$ to be fixed later on, and decompose $Q_\alpha f$ as
\begin{eqnarray}\label{eq-Q-decomp}
Q_\alpha f(x,y)
&=&
\frac1{\Gamma(\alpha)}\int_0^\delta t^{\alpha-1}(P_tf(x)-P_tf(y))dt
+\frac1{\Gamma(\alpha)}\int_\delta^\infty t^{\alpha-1}(P_tf(x)-P_tf(y))dt \notag\\
&\equiv&
Q^\delta_\alpha f(x,y)+\bar{Q}^\delta_\alpha f(x,y) .
\end{eqnarray}
We then estimate the right hand side of \eqref{eq-Q-decomp} term by term. For $Q^\delta_\alpha f(x,y)$ we resort to \eqref{eq-pt-bd} in order to get
\begin{equation*}
|Q^\delta_\alpha f(x,y)|
\le C\int_0^\delta t^{\alpha-1}\left(|P_tf(x)|+|P_tf(y)|\right)dt
\le C\| f \|_{L^2(\Heis,\mu)}\int_0^\delta t^{\alpha-1-\frac{n+1}{2}}dt.
\end{equation*}
Since we have chosen $\alpha>\frac{n+1}{2}$, we obtain
\begin{equation}\label{eq-Q-delta-ub}
|Q^\delta_\alpha f(x,y)|\le C\delta^{\alpha-\frac{n+1}{2}}\| f \|_{L^2(\Heis,\mu)}.
\end{equation}
As far as $\bar{Q}^\delta_\alpha f(x,y)$ in \eqref{eq-Q-decomp} is concerned, we use relation \eqref{eq-pt-cont}. This yields
\begin{align*}
|\bar{Q}^\delta_\alpha f(x,y)|&\le C \int_\delta^\infty t^{\alpha-1}|P_tf(x)-P_tf(y)|dt\\
&\le C\| f \|_{L^2(\Heis,\mu)}\left( \int_\delta^\infty t^{\alpha-2-n/2}{d_{cc}(x,y)}dt \right).
\end{align*}
Due to the fact that $\alpha<1+\frac{n}2$ we thus get
\begin{equation}\label{eq-Q-bar-ub}
|\bar{Q}^\delta_\alpha f(x,y)|\le C\delta^{\alpha-1-n/2}{d_{cc}(x,y)} \| f \|_{L^2(\Heis,\mu)}.
\end{equation}
In order to make \eqref{eq-Q-delta-ub} and \eqref{eq-Q-bar-ub} comparable, we choose $\delta$ such that $\delta^{\alpha-1-n/2}d_{cc}(x,y)=\delta^{\alpha-\frac{n+1}{2}}$, which easily gives  $\delta=d_{cc}(x,y)^{2}$. With this choice of $\delta$ and plugging \eqref{eq-Q-delta-ub}, \eqref{eq-Q-bar-ub} into \eqref{eq-Q-decomp} we obtain that 
\begin{equation}\label{a5}
\bigg|\int_\Heis \left(G_\alpha(x,q)-G_\alpha(y,q)\right) f(q) \, d\mu(q)\bigg|
\le
Cd_{cc}(x,y)^{2\alpha-(n+1)} \| f \|_{L^2(\Heis,\mu)}.
\end{equation}

Having proved \eqref{a5}, one can establish \eqref{a4} in the following way:
Let $K_{\alpha,x,y}$ be the linear form on $L^2(\Heis,\mu)$ corresponding to the integral kernel $G_\alpha(x,\cdot)-G_\alpha(y,\cdot)$. Clearly from the  above inequality we know that for each fixed $x$ and $y$, $K_{\alpha,x,y}: L^2(\Heis,\mu) \to \R$ is a continuous form. Moreover, denoting by $\|\cdot\|_{\mathrm{op}}$ the operator norm on $L^2(\Heis,\mu)$ and resorting to~\eqref{a5}, it holds that 
\begin{equation}\label{eq-K-norm}
\|G_\alpha(x,\cdot)-G_\alpha(y,\cdot)\|_{L^2(\Heis,\mu)}=\|K_{\alpha,x,y}\|_{\mathrm{op}}
\le Cd_{cc}(x,y)^{2\alpha-(n+1)}.
\end{equation}
The conclusion then follows.
\end{proof}

With Proposition \ref{prop-integrability-G} in hand, we can now turn to a description of our fractional noise. For $\alpha>n/2+1/2$, it will be based on stochastic integration with respect to a white noise measure. Let then  $\mathcal{K}$ denote the $\sigma$-field on $\Heis$ that is associated to the topology induced from the Carnot-Carath\'eodory distance \eqref{eq-cc-dist}, and call $\cb$ the Borel $\sigma$-field on $\R_+$. We consider a real-valued centered Gaussian random measure $W_\Heis: \cb\otimes\mathcal{K}\to L^2(\Omega, \mathcal{G}, \mathbb{P})$  with intensity $\la\otimes\mu$ on $\Heis$, i.e. $W_\Heis$ is a white noise such that
\begin{itemize}
\item $W_\Heis$ is a measure on $(\R_+\times\Heis, \cb\otimes\mathcal{K})$ almost surely
\item For any set $D\in\mathcal{K}$ of finite measure and $s<t$, $W_\Heis([s,t]\times D)$ is a real-valued centered Gaussian variable with variance given by $\mathbf{E}\left(W_\Heis([s,t]\times D)^2 \right)=(t-s)\mu(D)$.
\item For any collection of pairwise disjoint measurable sets $\{[s_j, t_j]\times D_j\}_{j\in \N}\in (\mathcal{B}\times \mathcal{K})^\N$ where $\mathcal{B}$ is the Borel sets, the random variables $W_\Heis([s_j,t_j]\times D_j)$, $j\in \N$ are independent.
\end{itemize}
Notice that the Gaussian measure $W_\Heis$ gives rise to an isonormal Gaussian family $\{ W_\Heis(f), f\in L^2(\R_+\times \Heis; dt\otimes \mu)\}$ with covariance function 
\begin{equation}\label{eq-cov-white}
\bE\left(W_\Heis(f) W_\Heis(g) \right)=\int_{\R_+\times\Heis}f(t,x)g(t,x)\mu(dx)dt.
\end{equation}
%We also introduce the notation
%\begin{equation}\label{eq-notation}
%W_\Heis(t,dx):=\int_0^t W_\Heis(ds,dx)
%\end{equation}
%where $W_\Heis(dt,dx)$ is a real-valued centered Gaussian variable with variance 
%\[
%\mathbf{E}\left(W_\Heis(dt,dx)^2 \right)=dt\, d\mu(x).
%\]
Otherwise stated, the covariance function of $W_\Heis$ is given by \eqref{eq-cov-0<s<1} with $\alpha=0$. Therefore $W_\Heis$ also possesses the following properties.
\begin{itemize}
\item Invariance by dilations (as in Proposition \ref{prop-dilation}). It holds that
\begin{equation}\label{eq-W-Heis-dilation}
\lbrace W_\Heis(f\circ \delta_\lambda), f\in L^2(\R_+\times \Heis; dt\otimes \mu)\rbrace\stackrel{\mathcal{D}}{=}\lbrace\lambda^{-(n+1)} W_\Heis(f), f\in L^2(\R_+\times \Heis; dt\otimes \mu)\rbrace
\end{equation}
\item Invariance by rotations (as in Proposition \ref{prop-rotation}). We have
\begin{equation}\label{eq-W-Heis-rotation}
\lbrace W_\Heis(f\circ R_\theta);  f\in L^2(\R_+\times \Heis; dt\otimes \mu)\rbrace\stackrel{\mathcal{D}}{=}\lbrace W_\Heis(f);   f\in L^2(\R_+\times \Heis; dt\otimes \mu))\rbrace.
\end{equation}
\item Invariance by left translation (as in Propostion \ref{prop-translation}). It is readily checked that
\begin{equation}\label{eq-W-Heis-translation}
\lbrace W_\Heis(f\circ L_x),  f\in L^2(\R_+\times \Heis; dt\otimes \mu)\rbrace\stackrel{\mathcal{D}}{=}\lbrace W_\Heis(f),  f\in L^2(\R_+\times \Heis; dt\otimes \mu)\rbrace.
\end{equation}
\end{itemize}
We also introduce an operator $\cg_\alpha$ which is a reformulation of the form $K_{\alpha, p,e}$ in \eqref{eq-K-norm} as an operator from $L^2(\Heis, \mu)$ to $L^2(\Heis, \mu)$.
\begin{definition}
Let $n/2+1/2<\alpha<n/2+1$ and recall that $G_\alpha$ is defined by \eqref{green function}. We define an operator $\cg_\alpha$ on $L^2(\Heis, \mu)$ by
\begin{equation}\label{eq-cg-s}
\cg_\alpha f(p)=\int_\Heis \left(G_\alpha(p,q)-G_\alpha(e,q)\right)f(q)d\mu(q),\quad p\in \Heis.
\end{equation}
Notice that formally, for $\mu$ almost all $p\in \Heis$ we have 
\[
\cg_\alpha f(p)=(-\Delta)^{-\alpha}f(p)-(-\Delta)^{-\alpha}f(e).
\]
\end{definition}

We are now ready to introduce the Gaussian field on $\R_+\times\Heis$ which is analogous to a Brownian motion in time and fractional Brownian motion in space.
\begin{proposition}\label{prop-W-pt}
For $n/2+1/2<\alpha<n/2+1$, recall that $\cg_\alpha$ is defined by \eqref{eq-cg-s}, and consider the white noise $W_\Heis$ introduced above. Then the process $\{\bW_{\!\al}(t,x); t\in\R_+, x\in\Heis\}$ is properly defined by 
\begin{equation}\label{eq-W-ptwise}
\bW_{\!\al}(t,x)=W_\Heis (G_\alpha(x, \cdot)-G_\alpha(e,\cdot))
=\int_{\R_+\times\Heis} \mathbbm{1}_{[0,t]}(r)(G_\alpha(x, q)-G_\alpha(e,q)) \, dW_\Heis(r,q) .
\end{equation}
One can also formally express $\bW_{\!\alpha}(t,x)$ as
\[
\bW_{\!\al}(t,x)=\cg_\alpha(W_\Heis(t,\cdot))(x).
\]
\end{proposition}
\begin{proof}
The proper definition of the right hand side of \eqref{eq-W-ptwise} directly stems from relations~\eqref{eq-cov-white} and~\eqref{a4}.
\end{proof}

Next we check  that Proposition \ref{prop-dilation} (invariance by dilations), Proposition \ref{prop-rotation} (invariance by rotations) and Proposition \ref{prop-translation} (invariance by left-translation) are still valid for $\bW_{\!\al}$ in the regime $n/2+1/2<\alpha<n/2+1$. Moreover, some pointwise versions of \eqref{eq-dilation} and \eqref{eq-rotation-d} are available. Let us start with the dilation invariance alluded to above.

\begin{proposition}\label{prop-dilation-pt}
Let $\bW_\alpha$  be the field described in Proposition \ref{prop-W-pt}, with $n/2+1/2<\alpha<n/2+1$.  For any $\lambda>0$, we have that
\begin{equation}\label{eq-dilation-pt}
 \{\bW_{\!\al}(t,\delta_\lambda x); t\ge0,x\in\Heis\}\stackrel{\mathcal{D}}{=}\lambda^{2\alpha-(n+1)}\{\bW_{\!\al}(t,x); t\ge0,x\in\Heis\}.
\end{equation}
\end{proposition}
\begin{proof}
By its definition \eqref{eq-W-ptwise} we know that
\begin{equation}\label{eq-dilation-pt-a}
\bW_{\!\al}(t,\delta_\lambda x)
=\int_0^t\int_\Heis (G_\alpha(\delta_\lambda x, q)-G_\alpha(e,q))dW_\Heis(s,q).
\end{equation}
Moreover, owing to \eqref{eq-dil-kernels} and the fact that $\delta_\lambda e=e$ we have
\[
G_\alpha(\delta_\lambda x, q)-G_\alpha(e,q)=\lambda^{2\alpha-2(n+1)}\left(G_\alpha( x, \delta_\lambda^{-1}q)-G_\alpha(e,\delta_\lambda^{-1}q)\right).
\]
Plug this identity into  \eqref{eq-dilation-pt-a} in order to get
\begin{equation}\label{eq-W-s-Heis-a}
\bW_{\!\al}(t,\delta_\lambda x)=\lambda^{2\alpha-2(n+1)}W_\Heis(g\circ \delta_\lambda^{-1}),
\end{equation} 
where the function $g$ is defined by
\[
g(s,x)=\mathbbm{1}_{[0,t]}(s)(G_\alpha(x,q)-G_\alpha(e,q)).
\]
We now invoke the invariance \eqref{eq-W-Heis-dilation} in \eqref{eq-W-s-Heis-a} which yields
\[
\bW_{\!\al}(t,\delta_\lambda x)
\stackrel{\mathcal{D}}{=}\lambda^{2\alpha-(n+1)}W_\Heis(g)=\lambda^{2\alpha-(n+1)}\bW_{\!\al}(t,x),
\]
where the last identity stems again from the definition \eqref{eq-W-ptwise} of $\bW_{\!\al}$. This concludes the proof of our claim \eqref{eq-dilation-pt}.
\end{proof}

In the proposition below we state the rotation invariance property of our noise, which is obtained similarly to the case $0<\alpha<n/2+1/2$. 
\begin{proposition}\label{prop-rotation-pt}
Let $n/2+1/2<\alpha<n/2+1$ and let $\bW_\alpha$ be the same random field as in Proposition~\ref{prop-dilation-pt}. For $\theta\in[0,2\pi)$, consider the rotation $R_\theta$ defined by \eqref{eq-rotation}. Then it holds that
\begin{equation}\label{eq-rotation-pt}
 \bW_{\!\al}(t,R_\theta (x))\stackrel{\mathcal{D}}{=}\bW_{\!\al}(t,x) \, ,
\end{equation}
for all $(t,x)\in \R_+\times \Heis$.
\end{proposition}
\begin{proof}
Note that $R_\theta(e)=e$. From \eqref{eq-W-ptwise} we have
\begin{equation}\label{eq-rotation-pt-a}
\bW_{\!\al}(t,R_\theta (x))
=\int_0^t\int_\Heis (G_\alpha(R_\theta (x), q)-G_\alpha(R_\theta(e),q))dW_\Heis(s,q) .
\end{equation}
Since the Green function is invariant under $R_\theta$, we have that
\[
G_\alpha(R_\theta (x), q)-G_\alpha(R_\theta(e),q)=G_\alpha(x, R_{-\theta}(q))-G_\alpha(e,R_{-\theta}(q)) \, .
\] 
Plug this identity into \eqref{eq-rotation-pt-a} and resort to the invariance \eqref{eq-W-Heis-rotation} of  $W_\Heis$ under $R_\theta$. We then obtain the desired conclusion.
\end{proof}

From its definition \eqref{eq-W-ptwise} one can easily observe that $\bW_\alpha$ is not invariant under left-translation, but its increment is. The proposition below verifies such property for $\bW_\alpha$. 
\begin{proposition}\label{prop-translation-pt}
Let $\bW_{\!\al}$ be the same random field as in Proposition~\ref{prop-W-pt}, with $n/2+1/2<\alpha<n/2+1$. For any $x\in\Heis$, let $L_x$ be the left translation defined by \eqref{eq-left-transl}. Then it holds that for any $t\in \R_+$ and $z\in \Heis$,
\begin{equation}\label{eq-rotation-pt}
 \{\bW_{\!\al}(t,L_x y)- \bW_{\!\al}(t,L_x z);y\in\Heis\}\stackrel{\mathcal{D}}{=}\{\bW_{\!\al}(t,y)-\bW_{\!\al}(t,z);y\in\Heis\}.
\end{equation}
\end{proposition}
\begin{proof}
From \eqref{eq-W-ptwise} we have
\begin{equation}\label{eq-translation-pt-a}
 \bW_{\!\al}(t,L_x y)- \bW_{\!\al}(t,L_x z)
=\int_0^t\int_\Heis (G_\alpha(xy, q)-G_\alpha(xz,q))dW_\Heis(s,q) .
\end{equation}
The conclusion then follows from the left translation invariance property of the green function and \eqref{eq-W-Heis-translation}. 
\end{proof}

At last let us address the H\"older continuity of the field $\bW$.
\begin{proposition}\label{prop-W-alpha_Lp}
Let $\bW_{\!\alpha}$  be the field described in Proposition \ref{prop-W-pt}, with $n/2+1/2<\alpha<n/2+1$. Then
\begin{enumerate}[wide, labelwidth=!, labelindent=0pt, label=\textnormal{(\roman*)}]
\setlength\itemsep{.1in}

\item 
For all $p\ge1$, there exists a constant $C_p$ such that for any $t\in \R_+$ and $x,y\in \Heis$,
\begin{equation}\label{eq-W-alpha_Lp}
\|\bW_{\al}(t,x)-\bW_{\al}(t,y)\|_{L^p(\Omega)}\le C_p d_{cc}(x,y)^{2\alpha-(n+1)}.
\end{equation}
\item 
For every compact set $K\subset \Heis$ and every $\epsilon>0$, the process $x\mapsto \bW_\al(t,x)$ is $(2\al-(n+1)-\epsilon)$-H\"older continuous on $K$.
\end{enumerate}
\end{proposition}
\begin{proof}
Since $\bW_\al$ is a centered Gaussian process, we only need to prove \eqref{eq-W-alpha_Lp} for $p=2$. Moreover, owing to \eqref{eq-W-ptwise} we have
\[
\bE\bigg((\bW_{\!\al}(t,x)-\bW_{\!\al}(t,y))^2\bigg)=\int_{\Heis} (G_\alpha(x, q)-G_\alpha(y,q))^2 d\mu(q).
\]
Therefore we get \eqref{eq-W-alpha_Lp} for $p=2$ thanks to a direct application of Proposition \ref{prop-integrability-G}. This proves item~(i) above. Our claim (ii) is then obtained by a standard use of Kolmogorov's criterion. 
\end{proof}

\begin{remark}
In Section \ref{sec:s-less-than-1} and \ref{sec:s-larger-than-1} we have respectively treated the cases
\[
0<\al<\frac{n+1}{2},\qquad\text{and} \qquad \frac{n+1}{2}<\al<\frac{n+2}2 .
\]
The situation $\alpha=\frac{n+1}2$ is thus ruled out from our study. It corresponds to a log-correlated type process (see e.g \cite{LSSW}). This kind of log-correlated process deserves a separate analysis and is avoided here for sake of conciseness. 
\end{remark}
\begin{remark}\label{Rmk: Hurst parameter}
The reference \cite{LSSW} is concerned with fractional processes of the form $(-\Delta)^{-\al} W$ in $\R^d$, where $W$ is a white noise. The corresponding Hurst parameter $H_{\R^d}$ is defined in  \cite{LSSW} as $H_{\R^d}=2\al-d/2$. Now according to our Proposition \ref{prop-W-alpha_Lp}, the Hurst parameter in $\Heis$ should be defined as
\[
H_\Heis=2\al-(n+1).
\]
Comparing with the formula for $H_{\R^d}$, the effective dimension of $\Heis$ for the problem at stake is $Q=2(n+1)$. As mentioned in the introduction, this has to be contrasted with the topological dimension $2n+1$ of $\Heis$, and notice that $d_{\rm{eff}}$ coincides with the Hausdorff dimension.
\end{remark}
\section{Existence and uniqueness of the It\^{o} solution}\label{existence and uniqueness}

In this section we establish existence and uniqueness for the solution of the stochastic heat equation~\eqref{eq:pam} interpreted in the It\^{o} sense. We shall explore two different methods in order to do so, namely: (i) Applying a more general result about nonlinear stochastic heat equations on Lie groups contained in \cite{PT}. (ii) Computing and controlling the chaos expansion for the linear equation~\eqref{eq:pam}.

\subsection{Existence and uniqueness through It\^o calculus}
The article~\cite{PT} relies on stochastic integration in infinite dimensions techniques, in order to solve 
nonlinear heat equations on Lie groups. In this section we will see how to apply this method to our specific equation~\eqref{eq:pam}.

We start by recalling some elements of It\^o stochastic integration in our context. 
In all the cases examined in Sections~\ref{sec:s-less-than-1} and~\ref{sec:s-larger-than-1}, our covariance function $\laa=G_{2s}$ in~\eqref{eq-cov} is  more regular than the Dirac delta-function. Therefore following the same arguments as in \cite{Wa}, one can extend our noise $\bW$ to a $\sigma$-finite $L^2$-valued measure $B\mapsto \bW_{\!\al}(B)$ defined for bounded Borel sets $B$ in $\mathbb{R}_{+}\times\Heis$. In particular we introduce 
\begin{align}\label{def: martingale measure}
M_t(A):=\bW_{\!\al}([0,t]\times A),\quad\text{for all}\quad 
t\ge 0, \, A\in\mathcal{B}_b(\Heis).
\end{align}
Let $\{\mathcal{F}_t, t\geq0\}$ be the filtration given by
$$\mathcal{F}_t:=\sigma\lcl M_s(A): 0\leq s\leq t, \, A\in\mathcal{B}_b(\Heis)\rcl 
\vee\mathcal{N},\qquad t\geq 0,$$
which is the natural filtration generated by $\bW$, augmented by all $\mathbb{P}$-null sets $\mathcal{N}$ in $\mathcal{F}$. We also consider the family of subsets of $\mathbb{R}_+\times\Heis\times\Omega$, which contains all sets of the form $\{0\}\times A\times F_0$ (where $A$ is a Borel subset of $\Heis$ and $F_0\in\mathcal{F}_0$), as well as $(s,t]\times A\times F$ with $0\leq s<t$ (where $F\in\mathcal{F}_s$). This class of sets is called the class of predictable rectangles. The $\sigma$-field generated by predictable rectangles is called the predictable $\sigma$-field, which is denoted by $\mathcal{P}$. Sets in $\mathcal{P}$ are called predictable sets. A random field $X$ is called predictable if $X$ is $\mathcal{P}$-measurable.

\subsubsection{Case of $\bW_{\!\al}$ with $\alpha\in(0,\frac{n+1}{2})$} Fix $\alpha\in(0,\frac{n+1}{2})$ and consider $\bW_{\!\al}$ defined by the covariance~\eqref{eq-cov-0<s<1}. We are now ready to introduce a reasonable class of It\^o integrable fields on $\R_{+}\times\Heis$. Namely for $p\geq2$, denote by $\mathcal{P}_p$ the set of all predictable random fields such that
\begin{align}\label{def: P_p}
\|f\|_{M,p}^2:=\int_{\mathbb{R}_{+}}ds\int_{\Heis\times\Heis}G_{2s}(q_1,q_2)\|f(s,q_1)f(s,q_2)\|_{\frac{p}{2}}\, \mu(dq_1)\mu(dq_2)<\infty,
\end{align}
where we recall that $\mu$ stands for the Haar measure on $\Heis$ and  $\|\cdot\|_p$ denotes the $L^p(\Omega)$-norm.  We have the following Proposition, which can be proved along the same lines as in \cite{Wa}.

\begin{proposition}
Suppose that for some $t>0$ and $p\geq2$, a random field $X=\{X(s,q): (s,q)\in(0,t)\times\Heis\}$ has the following properties,

\vspace{-.05in}

\begin{enumerate}[wide, labelwidth=!, labelindent=0pt, label=\textnormal{(\roman*)}]
\setlength\itemsep{.0in}
\item
 $X(s,\cdot)$ is adapted to $\mathcal{F}_s$;
\item
$X$ is jointly measurable with respect to $\mathcal{B}(\mathbb{R}_{+}\times\Heis)\otimes\mathcal{F}$;
\item
$\|X\|_{M,p}<\infty$, where the $(M,p)$-norm is defined by \eqref{def: P_p}.
\end{enumerate}
Then the field $X \1_{(0,t)}$ is in $\mathcal{P}_p$.
\end{proposition}

We now start the construction of a It\^o type stochastic integral with respect to $\bww_\alpha$. The definition below handles the case of elementary processes.
\begin{definition}\label{def: elem proc}
An elementary process $g$ is a process given by
$$
g(s,q)=\sum_{i=1}^n\sum_{j=1}^mX_{i,j}\mathbf{1}_{(a_i,b_i]}\mathbf{1}_{A_j},
$$
where $n$ and $m$ are finite positive integers, $-\infty<a_1<b_1<\cdots<a_n<b_n<\infty$, the sets $A_j$ are elements of $\mathcal{B}_b(\Heis)$ and $X_{i,j}\in\mathcal{F}_{a_i}$. The integral of such a process with respect to $\bW_{\!\al}$ is defined as
\begin{align}\label{isometry-stoch integral}
\int_{\mathbb{R}^+}\int_\Heis g(s,q)\bW_{\!\al}(ds,dq)=\sum_{i=1}^n\sum_{j=1}^mX_{i,j}\bW_{\!\al}({(a_i,b_i]}\times{A_j}),
\end{align}
where we recall that the quantities $\bW_{\!\al}([a_i,b_i]\times A_j)$ are introduced in \eqref{def: martingale measure}. Moreover, for the elementary process $g$ we have 
\begin{align}\label{Ito iso}
\bE\left[\left(\int_{\mathbb{R}^+}\int_\Heis g(s,q)\bW_{\!\al}(ds,dq)\right)^2\right]=\|f\|^2_{M,2},
\end{align}
where the norm $\|f\|_{M,2}$ is given by \eqref{def: P_p}.
\end{definition}

Predictable processes are the natural class of processes which can be integrated with respect to $\bW_{\!\al}$. In the proposition below we present the extension of It\^{o}'s integral to processes in $\mathcal{P}_2$. 
\begin{proposition}
Recall that the set $\mathcal{P}_2$ is defined by the norm \eqref{def: P_p}. Then the following holds true.

\begin{enumerate}[wide, labelwidth=!, labelindent=0pt, label={\textnormal{(\roman*)}}]
\setlength\itemsep{.05in}
\item 
The space of elementary processes defined in Definition \ref{def: elem proc} is dense in $\mathcal{P}_2$.
\item 
For $g\in\mathcal{P}_2$, the stochastic integral $\int_{\mathbb{R}^+}\int_\Heis g(s,x)\bW_{\!\al}(ds,dx)$ is defined as the $L^2(\Omega)$-limit of elementary processes approximating $g$, and \eqref{isometry-stoch integral} still holds true.
\end{enumerate}
\end{proposition}
\begin{proof}
The proof follows by standard arguments. See for example \cite{Da,Dalang-Quer}.
\end{proof}

\begin{remark}
With expression \eqref{def: P_p} in mind, it should be apparent that relation \eqref{Ito iso} is an extension of Proposition \ref{prop-field-W}-(1) to predictable processes. 
\end{remark}

With those preliminary notions in hand, our stochastic PDE  \eqref{eq:pam} is formally written in its mild form as follows,
\begin{align}\label{mild sol}
u(t,q)=J_0(t,q)+I(t,q),
\end{align}
where $p_t$ is the heat kernel in \eqref{eq-Heis-kernel},  $J_0(t,q)=\int_{\Heis}p_t(q^{-1}q_1)d\mu(q_1)$ is the solution to the homogeneous equation, and the stochastic integral is given by
\begin{align}\label{stochastic integral}
I(t,q):=\int_{[0,t]\times\Heis}p_{t-s}(q^{-1}q_1)u(s,q_1)\bW_{\!\al}(ds,dq_1).
\end{align}
We also define more precisely what we mean by It\^o solution to \eqref{eq:pam}.
\begin{definition}\label{def of Ito sol}
A process $u=\{u(t,q), (t,q)\in\mathbb{R}_{+}\times\Heis\}$ is called a random field solution of~\eqref{eq:pam} if the following conditions are met:
\begin{enumerate}[wide, labelwidth=!, labelindent=0pt, label={\textnormal{(\arabic*)}}]
\setlength\itemsep{.02in}
\item $u$ is adapted;

\item $u$ is jointly measurable with respect to $\mathcal{B}(\mathbb{R}_{+}\times\Heis)\times\mathcal{F}$;
\item $\|I(t,q)\|_2<\infty$ for all $(t,q)\in\mathbb{R}_{+}\times\Heis$;
\item The function $(t,q)\to I(t,q)$ is continuous in $L^2(\Omega)$;
\item $u$ satisfies \eqref{mild sol} almost surely for all $(t,q)\in\mathbb{R}_{+}\times\Heis$.
\end{enumerate}
\end{definition}

We can now define rigorously and give a sufficient condition allowing to solve the stochastic heat equation \eqref{eq:pam}.

\begin{theorem}\label{prop:ito-exist-unique}
Fix a regularity parameter $\alpha\in(n/2,(n+1)/2)$ and let $\bW_{\!\al}$ be defined by \eqref{eq-cov-0<s<1}-\eqref{eq-d}. Then there exists a unique solution to \eqref{eq:pam} in the It\^{o} sense, interpreted as in Definition~\ref{def of Ito sol}.
\end{theorem}
\begin{proof}

We shall apply a general result proved in \cite{PT}. Notice that \cite{PT} was considering equation \eqref{eq:pam} interpreted in the infinite dimensional setting. However, the general considerations in \cite{Dalang-Quer} apply to our setting. Therefore, one can identify the random field solution \eqref{mild sol} and the solution constructed in \cite{PT}. 

The sufficient condition given in \cite{PT} can be spelled out as
\begin{align}\label{sufficient cond for exist}
\int_{B(e,1)}d_{cc}(e,q)^{-2n}\Lambda(q)d\mu(q)<\infty.
\end{align}
where $\Lambda$ is featuring in relation \eqref{eq-cov}. For the noise $\bW_{\!\al}$ we thus have $\Lambda=G_{2\alpha}$ as in \eqref{eq-d}. Hence we are reduced to show \eqref{sufficient cond for exist} when  {$0<\alpha<(n+1)/2$}. To this aim we invoke relation~\eqref{eq-kernel-cc}, which allows to write
\begin{align}\label{eq-soln-ub}
\int_{B(e,1)}d_{cc}(e,q)^{-2n}\Lambda(q)d\mu(q)\lesssim\int_{B(e,1)}\frac{d\mu(q)}{d_{cc}(e,q)^{4n+2-4\alpha}}.
\end{align}
Using \cite{MR1942237} we write the integral \eqref{eq-soln-ub} in polar coordinates with respect to $r=d_{cc}(e,q)$. This brings about a Jacobian term $r^{Q-1}$, where $Q=2n+2$ is the homogeneous dimension of $\Heis$ alluded to in Remark \ref{Rmk: Hurst parameter}. We end up with 
\begin{align*}
\int_{B(e,1)}d_{cc}(e,q)^{-2n}\Lambda(q)d\mu(q)\lesssim\int_0^1{r^{2n+1-4n-2+4\al}}dr =\int_0^1\frac{dr}{r^{2n+1-4\alpha}}.
\end{align*}
The latter integral is finite whenever $\alpha>n/2$, which finishes our proof.
\end{proof}

\subsection{Existence and uniqueness through chaos expansions} 

In this section we take another look at equation~\eqref{eq:pam}. Namely we shall write the solution to this equation directly as a random field, by characterizing its chaos expansion. The advantage of this approach is twofold: first it enables to state necessary and sufficient conditions on the covariance function $\laa$ in order to get existence and uniqueness result for the stochastic heat equation. Then we shall also get some valuable information about moments of the solution.

\subsubsection{Preliminaries on  chaos expansions}
In this section we recall the minimal amount of Malliavin calculus tools necessary to state our results. The reader is referred to~\cite{Nu} for more details.

Recall that the Cameron-Martin space $\mathcal{H}$ is defined in \eqref{def mathscrH}. The $m$-th Wiener chaos, denoted by $\mathcal{H}_\m$, is defined as the closed linear span of the random variables of the form $H_\m(\bW_{\!\al}(\varphi))$, where $\varphi$ is an element of $\mathcal{H}$ with norm one and $H_\m$ is the $\m$-th Hermite polynomial. We denote by $I_\m$ the linear isometry between $\mathcal{H}^{\otimes \m}$ (equipped with the modified norm $\sqrt{\m!}\|\cdot\|_{\mathcal{H}^{\otimes \m}}$) and the $\m$-th Wiener chaos $\mathcal{H}_\m$. It is given by $I_\m(\varphi^{\otimes \m})=\m!H_\m(\bW_{\!\al}(\varphi))$, for any $\varphi\in\mathcal{H}$ with $\|\varphi\|_{\mathcal{H}}=1$. 

Any square integrable random variable, which is measurable with respect to the $\sigma$-field generated by $\bW_{\!\al}$, has an orthogonal Wiener chaos expansion of the form
$$F=\bE(F)+\sum_{\m=1}^\infty I_\m(f_\m),$$
where $f_\m$ are symmetric elements of $\mathcal{H}^{\otimes \m}$, uniquely determined by $F$.
This kind of expansion can easily be extended to random fields. Specifically 
consider a field $u=\{u(t,q), t\geq0, q\in\Heis\}$ such that $\bE [u(t,q)^2]<\infty$ for all $t,q$. Then $u(t,q)$ has a Wiener chaos expansion of the form
\begin{align}\label{u chaos expansion}
u(t,q)=\bE[ u(t,q)]+\sum_{\m=1}^\infty I_\m(f_\m(\cdot, t,q)),
\end{align}
where the series converges in $L^2(\Omega)$. With this kind of decomposition in hand, we can now give a definition of Skorohod integrable field.

\begin{definition}\label{def:sko-integral} We say that the random field $u$ with decomposition~\eqref{u chaos expansion} is Skorohod integrable if the deterministic function $\bE[ u]$ sits in the space $\mathcal{H}$, if $f_m$ (considered as a function on $(\R_{+}\times\Heis)^{\m+1}$) is an element of $\mathcal{H}^{\otimes (\m+1)}$ for all $\m\geq1$, and if the following series converges in $L^2(\Omega)$:
$$\bW_{\!\al}(\bE(u))+\sum_{\m=1}^\infty I_{\m+1}(\tilde{f}_\m),
$$
where $\tilde{f}_\m$ denotes the symmetrization of $f_\m$. We will denote the sum of this series by $\delta(u)=\int_0^\infty\int_{\Heis}u \, \delta\bW_{\!\al},$ which we call the Skorohod integral of $u$.
\end{definition}

Note that when $u$ is  the solution to equation \eqref{mild sol} given by Theorem \ref{prop:ito-exist-unique}, the It\^o type stochastic integral \eqref{stochastic integral} is equivalent to a Skorohod integral as introduced in Definition~\ref{def:sko-integral}. Hence by a standard iteration procedure (see e.g. \cite{HN09}), $u$ admits a chaos expansion as in~\eqref{u chaos expansion} with $f_\m$ given explicitly by
\begin{equation}\label{f_n explicit}
f_\m(s_1,q_1,\cdots,s_\m,q_\m, t,q)
=\frac{1}{\m!}p_{t-s_\sigma(\m)}(q,q_{\sigma(\m)})\cdots p_{s_{\sigma(2)}-s_{\sigma(1)}}(q_{\sigma(2)},q_{\sigma(1)})p_{s_{\sigma(1)}}u_0(q_{\sigma(1)}),
\end{equation}
where $\sigma$ denotes the permutation of $\{1,2,\dots,\m\}$ such that $0<s_{\sigma(1)}<\cdots<s_{\sigma(\m)}<t$. 
We can thus state the following existence-uniqueness result, for which more details are provided e.g. in \cite{HHNT, HN09}.

\begin{proposition}\label{general existence chaos}
Equation \eqref{mild sol} admits a unique solution in the It\^{o}-Skorohod sense (that is when $I(t,q)$ in \eqref{stochastic integral} is understood as in Definition \ref{def:sko-integral}) if and only if the following holds:
\begin{align}\label{exist of sol-chaos expansion}
\sum_{\m=0}^\infty \m!\|f_\m(\cdot, t,q)\|_{\mathcal{H}^{\otimes \m}}^2<\infty.
\end{align}
\end{proposition}

\noindent
In the sequel we will give conditions on $\alpha$ so that \eqref{exist of sol-chaos expansion} is fulfilled. 

\subsubsection{Estimates for a fixed chaos}Having \eqref{exist of sol-chaos expansion} in mind, we now proceed to upper bound each contribution $\|f_\m(\cdot, t,q)\|_{\mathcal{H}^{\otimes \m}}$. Since $f_k$ is given by \eqref{f_n explicit}, we first introduce a notation valid for  $0< s_1<\dots< s_\m<t$,
\begin{align}\label{def: g chaos}g({s},{q},t,{y})=p_{t-s_\m}(y,q_\m)\cdots p_{s_2-s_1}(q_2,q_1),\end{align}
where we use the convention  $q=(q_1,\ldots,q_\m)$ and $s=(s_1,\ldots,s_\m)$. We can now state a first upper bound for $\|f_\m(\cdot, t,q)\|_{\mathcal{H}^{\otimes \m}}$. 

\begin{lemma}
For $k\geq1$, let $f_k$ be the function defined by \eqref{f_n explicit}. Then for $t>0$ and $q\in\Heis$ we have 
\begin{align}\label{chaos bound}
\|f_\m(\cdot, t,q)\|^2_{\mathcal{H}^{\otimes \m}}\lesssim \|u_0\|^2_\infty\left[\frac{1}{k!}\left(\frac{2^{n-1}}{\pi^{n+1}}\right)^kM_{n,k}\right],
\end{align}
where the quantity $M_{n,k}$ is given by (note that we use the convention $s_{k+1}=t$ below):
\begin{align}
M_{n,k}=\int_{[0,t]^k_<}ds\,\prod_{i=1}^k\int_{\mathbb{R}}|\lambda_i|^n \sum_{m_i\in\mathbb{N}^n} |\lambda_i|^{-2\alpha}(2|m_i|+n)^{-2\alpha}e^{-8(s_{i+1}-s_{i})|\lambda_i|(2|m_i|+n)}d\lambda_i  .
\label{Chaos Fourier mode1}
\end{align}
%where we used the convention $s_{k+1}=t$.
\end{lemma}

\begin{proof}
We start with a couple of easy simplifications. First we have assumed that $u_0$ is a bounded function on $\Heis$. Therefore $p_{s_{\sigma(1)}}u_0$ in the right hand-side of \eqref{f_n explicit} is bounded.  Also note that by symmetry, we only need to evaluate for a particular time simplex $[0,t]_<^\m:=\{(s_1,\ldots,s_\m); 0< s_1<\dots< s_m<t\}$.
So, recalling the definitions of $f_k$ and $g$ in \eqref{f_n explicit} and~\eqref{def: g chaos} respectively, we obtain 
 $$
 \|f_\m(\cdot, t,q)\|^2_{\mathcal{H}^{\otimes \m}}\leq \frac{1}{k!}M_k(t,y) \, \|u_0\|_\infty^2,
 $$ where we have set
\begin{align}\label{H-norm g}
M_k(t,y)=\int_{[0,t]_<^\m}ds\int_{(\Heis)^{2\m}}g(s,q,t,y)\prod_{i=1}^\m G_{2\alpha}(q_i,q_i')g(s,q',t,y)
\, d\mu(q)d\mu(q').
\end{align}
In the sequel we will prove the bound claimed in \eqref{chaos bound} by proving a suitable bound for the quantity $M_k(t,y)$.

In order to compute $M_k(t,y)$ more explicitly, we first look at the integral with respect to $q_1$ and $q_1'$ in the right hand-side of \eqref{H-norm g}. We get an integral of the function
\begin{align}\label{def hat M_1}
\hat{M}_1(t,y,y')=\int_{(\Heis)^2}p_t(y,q)G_{2\alpha}(q,q')p_t(y',q')\, d\mu(q)d\mu(q'),
\end{align}
for a given $y\in\Heis$. By translation invariance, $p_t(y,q)$ takes the form $p_t(y^{-1}q)$. Hence the above integral becomes
\begin{align*}
\hat{M}_1(t,y,y')=\int_{(\Heis)^2}p_t(y^{-1}q)G_{2\alpha}(q,q')p_t((y')^{-1}q') \, d\mu(q)d\mu(q').
\end{align*}
Moreover, since $G_{2\alpha}$ is the kernel for the operator $(-\Delta)^{-2\alpha}$ (see identity \eqref{green function}), we get
\begin{eqnarray}\label{upper bound m=1 2}
\hat{M}_1(t,y,y')
&=&\int_{\Heis}p_t(y^{-1}q)\left((-\Delta)^{-2\alpha}p_t((y')^{-1}\cdot)\right)(q) \, d\mu(q)
\notag\\
&=&\int_{\Heis}\left((-\Delta)^{-\alpha}p_t(y^{-1}\cdot)\right)(q)\left((-\Delta)^{-\alpha}p_t((y')^{-1}\cdot)\right)(q)d\mu(q),
\end{eqnarray}
where we resort to the self-adjointness of $(-\Delta)^{-\alpha}$ for the second identity (see also \eqref{covariance by Laplacian}-\eqref{eq-d}). Therefore one can simply invoke Schwarz inequality in the right hand-side of \eqref{upper bound m=1 2}, which yields
\begin{align}
\hat{M}_1(t,y,y')&\leq\|(-\Delta)^{-\alpha}p_t(y^{-1}\cdot)\|_{L^2(\Heis,\mu)}\cdot 
\|(-\Delta)^{-\alpha}p_t((y')^{-1}\cdot)\|_{L^2(\Heis,\mu)}\nonumber\\
&\leq\sup_{x\in\Heis}\|(-\Delta)^{-\alpha}p_t(x^{-1}\cdot)\|_{L^2(\Heis,\mu)}^2\, .\label{Schwartz for hat M_1}
\end{align}
We now prove that the $L^{2}$ norm in the right hand side of~\eqref{Schwartz for hat M_1} does not depend on $x$. Namely notice that the Fourier transform of $p_t(x^{-1}\cdot)$ can be expressed thanks to~\eqref{eq-Fourier-transf} as
\begin{equation*}
\mathcal{F}p_t(x^{-1}\cdot)(\lambda)
=\int_{\Heis}p_t(x^{-1}q) U^\lambda_q \, d\mu(q)
=\int_{\Heis}p_t(q')U_{xq'}^\lambda \, d\mu(q'),
\end{equation*}
where we have used the change of variable $x^{-1}q=q'$ and the left invariance of $\mu$ for the second identity. Hence invoking the flow property of $U$ we obtain
\begin{equation*}
\mathcal{F}p_t(x^{-1}\cdot)(\lambda)
=\int_{\Heis}p_t(q')U_{x}^\lambda U_{q'}^\lambda \, d\mu(q')
=U_x^\lambda\circ\mathcal{F}p_t(\cdot)(\lambda).
\end{equation*}
Since $U_x^\lambda$ is a unitary operator, it does not change the Hilbert-Schmidt norm of $\mathcal{F}p_t(\lambda)$. Otherwise stated, we have
\begin{equation*}
\| \mathcal{F}p_t(x^{-1}\cdot)(\lambda) \|_{\textsc{hs}}
=
\| \mathcal{F}p_t(\lambda) \|_{\textsc{hs}},
\quad\text{for all } x\in\Heis, \, \la\in\R.
\end{equation*}
 Combining this observation with Plancherel's identity \eqref{eq:plancherel-operator}, the right hand-side of \eqref{Schwartz for hat M_1} becomes
\begin{align}\label{remove effect of shift}
\sup_{x\in\Heis}\|(-\Delta)^{-\alpha}p_t(x^{-1}\cdot)\|_{L^2(\Heis,\mu)}^2=\|(-\Delta)^{-\alpha}p_t\|_{L^2(\Heis,\mu)}^2.
\end{align}
%where the second identity stems from the left invariance of the Haar measure $\mu$.
Moreover,  making successive use of the projective Plancherel inequality  \eqref{proj Plancherel}, relation \eqref{Fourier transform Laplace} for the Fourier transform of $(-\Delta)^{-\alpha}f$ and \eqref{HeatKernelFourier} for $\hat{p}_t$, we get
\begin{align}
\|(-\Delta)^{-\alpha}p_t\|_{L^2(\Heis)}^2
=&\frac{2^{n-1}}{\pi^{n+1}}\sum_{m,l \in\mathbb{N}^n} \int_{\mathbb{R}}|\widehat{\left(-\Delta)^{-\alpha}p_t\right)}(m,l,\lambda)|^2|\lambda|^n d\lambda\nonumber\\
\leq&\frac{2^{n-1}}{\pi^{n+1}}\sum_{m\in\mathbb{N}^n}\int_{\mathbb{R}} |\lambda|^n |\lambda|^{-2\alpha}(2|m|+n)^{-2\alpha}e^{-8t|\lambda|(2|m|+n)}d\lambda.\label{est Fourier 1}
\end{align}
We can now plug  \eqref{est Fourier 1} into \eqref{remove effect of shift} and then back in \eqref{Schwartz for hat M_1}. This enables to write  
\begin{align}
\hat{M}_1(t,y,y')\leq\frac{2^{n-1}}{\pi^{n+1}}\sum_{m\in\mathbb{N}^n}\int_{\mathbb{R}} |\lambda|^n |\lambda|^{-2\alpha}(2|m|+n)^{-2\alpha}e^{-8t|\lambda|(2|m|+n)}d\lambda.\label{est Fourier 2}
\end{align}
Recalling that $\hat{M}_1$ is defined by \eqref{def hat M_1}, we can now report \eqref{est Fourier 2} into \eqref{H-norm g} and iterate the procedure over variables $q_i$. We end up with
\begin{multline}
M_k(t,y)
\leq\left(\frac{2^{n-1}}{\pi^{n+1}}\right)^k \\
\times
\int_{[0,t]^k_<}ds\,\prod_{i=1}^k\int_{\mathbb{R}}|\lambda_i|^n \sum_{m_i\in\mathbb{N}^n} |\lambda_i|^{-2\alpha}(2|m_i|+n)^{-2\alpha}e^{-8(s_{i+1}-s_{i})|\lambda_i|(2|m_i|+n)}d\lambda_i,\label{Chaos Fourier mode}
\end{multline}
which is exactly our claim \eqref{chaos bound}.
\end{proof}

\begin{remark}\label{rmk on alpha} At a heuristic level, one can see from \eqref{Chaos Fourier mode1} why the condition $\frac{n}{2}<\alpha<\frac{n+1}{2}$ pops out in order to solve equation \eqref{mild sol}. Indeed, the following holds true for the right hand-side of \eqref{Chaos Fourier mode1}:
\begin{enumerate}[wide, labelwidth=!, labelindent=0pt, label={(\roman*)}]
\setlength\itemsep{.1in}
\item In order for the integral with respect to $\lambda$ to be finite near $0$, one needs $n-2\alpha>-1$, that is $\alpha<\frac{n+1}{2}$.

\item In order for the integral with respect to $s$ first and then with respect to $\lambda$ to be finite, one needs $(n-1)-2\alpha<-1$, that is $\alpha>\frac{n}{2}$.

\item In order for the integral with respect to $s$ first and then with respect to $m$ to be finite, one needs $2\alpha+1>n$, that is $\alpha>\frac{n-1}{2}$.
\end{enumerate}
The above constrains can be combined to be $\frac{n}{2}<\alpha<\frac{n+1}{2}$, which  agrees with the condition in Theorem \ref{prop:ito-exist-unique}.
In the remainder of the section we will make those statements rigorous.
\end{remark}

With Remark \ref{rmk on alpha} in mind, let us assume that $\frac{n}{2}<\alpha<\frac{n+1}{2}$. Our next aim will be to get a proper upper bound on the right hand side of \eqref{Chaos Fourier mode1}. To this end,  we split the contribution of $d\lambda$ in the following way: fix $N\geq1$ and set
\begin{align}\label{def: C_N D_N}
D^+_N=\int_{|\lambda|\geq N} |\lambda|^{n-2\alpha-1}d\lambda, \quad\mathrm{and}\quad  D^-_N=\int_{|\lambda|<N}|\lambda|^{n-2\alpha}d\lambda.
\end{align}
By the assumption that $\alpha>\frac{n}{2}$, $D_N^+$ is finite for all positive $N$ and approaches to $0$ as $N\uparrow\infty$. Furthermore, the following two quantities are also finite whenever $\alpha>\frac{n}{2}$,
\begin{align}
\label{def: C_1 C_2} C_1=\sum_{m\in\mathbb{N}^n}(2|m|+n)^{-2\alpha}, \quad\mathrm{and}\quad C_2=\frac{1}{8}\sum_{m\in\mathbb{N}^n}(2|m|+n)^{-(2\alpha+1)}.
\end{align}
With this additional piece of notation in hand, we can now state our main technical lemma toward the evaluation of \eqref{Chaos Fourier mode1}.

\begin{lemma}\label{est: chaos Fourier mode}Let $\frac{n}{2}<\alpha<\frac{n+1}{2}$. For any $N>0$, let $D^+_N$, $D^-_N$  and $C_1$, $C_2$ be given in \eqref{def: C_N D_N}  and \eqref{def: C_1 C_2}. Consider the quantity $M_{n,k}$ defined by \eqref{Chaos Fourier mode1}. Then, we have
\begin{align}\label{final bound M_nk}
%&\int_{[0,t]^k_<}ds\,\prod_{i=1}^k\int_{\mathbb{R}}|\lambda_i|^n \sum_{m_i\in\mathbb{N}^n} |\lambda_i|^{-2\alpha}(2|m_i|+n)^{-2\alpha}e^{-8(s_{i+1}-s_{i})|\lambda_i|(2|m_i|+n)}d\lambda_i\\
M_{n,k}\leq\sum_{j=1}^k{k\choose j}  \frac{(tC_1D^-_N)^{j}}{j!}\left({C_2D^+_N}\right)^{k-j}.
\end{align}
\end{lemma}

\begin{proof}
In the right hand-side of \eqref{Chaos Fourier mode1}, we start with an elementary change of variables $s_{i+1}-s_{i}=w_i$ for $1\leq i\leq k-1$, and $t-s_k=w_k$. Denoting $dw=dw_1dw_2\cdots dw_k$, we have
\begin{align}\label{M_nk 1}
%\int_{[0,t]^k_<}ds\,\prod_{i=1}^k\int_{\mathbb{R}}|\lambda_i|^n \sum_{m_i\in\mathbb{N}^n} |\lambda_i|^{-2\alpha}(2|m_i|+n)^{-2\alpha}e^{-8(s_{i+1}-s_{i})|\lambda_i|(2|m_i|+n)}d\lambda_i\\
M_{n,k}\leq \int_{S_{t,k}}dw\,\prod_{i=1}^k\int_{\mathbb{R}}|\lambda_i|^n \sum_{m_i\in\mathbb{N}^n} |\lambda_i|^{-2\alpha}(2|m_i|+n)^{-2\alpha}e^{-8w_i|\lambda_i|(2|m_i|+n)}d\lambda_i,
\end{align}
where the set $S_{t,k}$ is defined by 
\begin{align}\label{set S_tk}
S_{t,k}=\{(w_1,\dots,w_k)\in [0,\infty)^k: w_1+\cdots+w_k\leq t\}.
\end{align}
Our next elementary step is to linearize the integrals in the right hand-side of \eqref{M_nk 1}. We get
%  Let $I$ be a subset of $\{1,2,\dots,k\}$ and $I^c=\{1,\dots,k\}\backslash I.$ Then
\begin{align*}
%&\int_{S_{t,k}}dw\,\prod_{i=1}^k\int_{\mathbb{R}}|\lambda_i|^n \sum_{m_i\in\mathbb{N}^n} |\lambda_i|^{-2\alpha}(2|m_i|+n)^{-2\alpha}e^{-8w_i|\lambda_i|(2|m_i|+n)}d\lambda_i\\
M_{n,k}\leq\sum_{m_i\in\mathbb{N}^n, \, i=1,\ldots,k}\int_{S_{t,k}}dw\int_{\mathbb{R}^k}\prod_{i=1}^k |\lambda_i|^{n-2\alpha}(2|m_i|+n)^{-2\alpha}e^{-8w_i|\lambda_i|(2|m_i|+n)}d\lambda_1\cdots d\lambda_k.
\end{align*}
Now recall that we have used a parameter $N\geq1$ in \eqref{def: C_N D_N}. We split the integrals with respect to $\lambda$ above according to this parameter. This yields
\begin{multline}\label{M_nk2}
M_{n,k}\leq\sum_{m_i\in\mathbb{N}^n, \, i=1,\ldots,k} \int_{S_{t,k}}dw\,\int_{\mathbb{R}^k} \prod_{i=1}^k 
\left( \mathbf{1}_{\{|\lambda_i|< N\}}+\mathbf{1}_{\{|\lambda_i|\geq N\}} \right)\\
\cdot |\lambda_i|^{n-2\alpha}(2|m_i|+n)^{-2\alpha}e^{-8w_i|\lambda_i|(2|m_i|+n)}d\lambda_1\cdots d\lambda_k.
\end{multline}
In order to handle the products in \eqref{M_nk2}, let us introduce an additional notation. Namely we denote by $I$ a generic subset of $\{1,\dots,k\}$ and $I^c=\{1,\dots,k\}\backslash I$ stands for its complement. We split $\R^k$ accordingly as $\R^k=\R^{|I|}\times\R^{|I^c|}$ for the integration with respect to $\lambda$, with related variables called $\lambda_I$ and $\lambda_{I^c}$. Then starting from \eqref{M_nk2} we get
\begin{align*}
M_{n,k}\leq\sum_{I\subset\{1,2,...,k\}}\sum_{m_i\in\mathbb{N}^n\atop i=1,...,k} \int_{S_{t,k}}Q_I(m,w)Q_{I^c}(m,w)dw,
\end{align*}
where we have set
\begin{eqnarray*}
Q_I(m,w)&=&
\int_{\mathbb{R}^{|I|}}\prod_{i\in I}\mathbf{1}_{\{|\lambda_i|< N\}} |\lambda_i|^{n-2\alpha}(2|m_i|+n)^{-2\alpha}e^{-8w_i|\lambda_i|(2|m_i|+n)}d\lambda_I. \\
Q_{I^c}(m,w)&=&
\int_{\R^{|I^c|}} \prod_{i\in I^c}\mathbf{1}_{\{|\lambda_i|\geq N\}} |\lambda_i|^{n-2\alpha}(2|m_i|+n)^{-2\alpha}e^{-8w_i|\lambda_i|(2|m_i|+n)}d\lambda_{I^c}.
\end{eqnarray*}
One can also factorize the integration with respect to $w$ in the following way: recalling the definition \eqref{set S_tk} of $S_{t,k}$, we trivially have $S_{t,n}\subset S_t^I\times S_t^{I^c},$ with $S_t^I=\{(w_i, i\in I): w_i\geq 0, \sum_{i\in I}w_{i}\leq t\}$ and $S_t^{I^c}$ defined similarly. Next bound the exponential terms  $e^{-8w_i|\lambda_i|(2|m_i|+n)}$ by 1 whenever $i\in I$. we end up with
\begin{align}\label{M_nk3}
M_{n,k}\leq\sum_{I\subset\{1,2,...,k\}}\sum_{m_i\in\mathbb{N}^n\atop i=1,...,k} \hat{Q}_{I}(m_I)\hat{Q}_{I^c}(m_{I^c}),
\end{align}
where we define $m_I=\{m_i; i\in I\}$, $m_{I^c}=\{m_i; i\in I^c\}$ and 
\begin{eqnarray}
\hat{Q}_I(m_I)&=&
\int_{S_t^I}\int_{\mathbb{R}^{|I|}}\prod_{i\in I}\mathbf{1}_{\{|\lambda_i|< N\}} |\lambda_i|^{n-2\alpha}(2|m_i|+n)^{-2\alpha}dw_Id\lambda_I,\label{def: hat Q_I} \\
\hat{Q}_{I^c}(m_{I^c})&=&
\int_{S_t^{I^c}}\int_{\mathbb{R}^{|I^c|}}\prod_{i\in I^c}\mathbf{1}_{\{|\lambda_i|\geq N\}} |\lambda_i|^{n-2\alpha}(2|m_i|+n)^{-2\alpha}e^{-8w_i|\lambda_i|(2|m_i|+n)}dw_{I^c}d\lambda_{I^c}.\label{def: hat Q_I^c}
\end{eqnarray}
Notice that each $m_i$ is an element of $\mathbb{N}^n$, and we will thus consider $m_I$ as an element of $\mathbb{N}^{n|I|}$. In the same way $m_{I^c}$ is understood as an element of $\mathbb{N}^{n|I^c|}$.
In addition, the expression \eqref{M_nk3} can be further factorized in order to get
\begin{align}\label{est: chaos mid step}
M_{n,k}\leq \sum_{I\subset\{1,2,...,k\}} \left(\sum_{m_I\in\mathbb{N}^{n|I|}}\hat{Q}_I(m_I)\right) \left(\sum_{m_{I^c}\in\mathbb{N}^{n|I^c|}}\hat{Q}_{I^c}(m_{I^c})\right).
\end{align}
We will now bound the two terms in the right hand-side of \eqref{est: chaos mid step} separately.

The term $\hat{Q}_I$ defined by \eqref{def: hat Q_I} is upper bounded as follows: we recall that $C_1$ is defined by \eqref{def: C_1 C_2}. Then since $\alpha<\frac{n+1}{2}$, we also resort to expression \eqref{def: C_N D_N} for $D_N^-$ plus an elementary integration over the simplex $S_t^I$, which yields 
%Recall the definition of $C_1$ and $C_2$ in \eqref{def: C_1 C_2}. Since $\frac{n}{2}<\alpha<\frac{n+1}{2}$, it is clear that
\begin{align}\label{est: chaos small lambda}
%\sum_{m_i\in\mathbb{N}^n\atop i\in I}\int_{S_t^I}\int_{\mathbb{R}^{|I|}}\prod_{i\in I}\mathbf{1}_{|\lambda_i|< N} |\lambda_i|^{n-2\alpha}(2|m_i|+n)^{-2\alpha}dw_Id\lambda_I
\sum_{m_I\in\mathbb{N}^{n|I|}} \hat{Q}_I(m_I)\leq \frac{(tC_1D^-_N)^{|I|}}{|I|!}.
\end{align}
Notice that the factor $t^{|I|}/|I|!$ in \eqref{est: chaos small lambda} comes from the integral over $dw_I$.

As far as the term $\hat{Q}_{I^c}$ defined by \eqref{def: hat Q_I^c} is concerned, we simply invoke the fact that $S_t^{I^c}\subset [0,\infty]^{|I^c|}$ and we integrate the exponential terms $\exp\left(-8w_i|\lambda_i|(2|m_i|+n)\right)$ over $w_i$ in order to get
\begin{align*}
&\sum_{m_{I^c}\in\mathbb{N}^{n|I^c|}}\hat{Q}_{I^c}(m_{I^c})\\
\leq&\sum_{m_{I^c}\in\mathbb{N}^{n|I^c|}}\int_{[0,\infty]^{|I^c|}}\int_{\mathbb{R}^{|I^c|}}\prod_{i\in I^c}\mathbf{1}_{\{|\lambda_i|\geq N\}} |\lambda_i|^{n-2\alpha}(2|m_i|+n)^{-2\alpha}e^{-8w_i|\lambda_i|(2|m_i|+n)}dw_{I^c}d\lambda_{I^c}\\
=&\frac{1}{8}\sum_{m_{I^c}\in\mathbb{N}^{n|I^c|}}\int_{\mathbb{R}^{|I^c|}}\prod_{i\in I^c}\mathbf{1}_{\{|\lambda_i|\geq N\}} |\lambda_i|^{n-2\alpha}(2|m_i|+n)^{-2\alpha}\frac{1}{{|\lambda_i|(2|m_i|+n)}}d\lambda_{I^c}.
\end{align*}
Owing to the fact that $\alpha>\frac{n}{2}$ and having the definitions \eqref{def: C_N D_N}-\eqref{def: C_1 C_2} in mind, we get
\begin{align}
\sum_{m_{I^c}\in\mathbb{N}^{n|I^c|}}\hat{Q}_{I^c}(m_{I^c})\leq\left({C_2D^+_N}\right)^{|I^c|}.\label{est: chaos large lambda}
\end{align}
Substituting \eqref{est: chaos small lambda} and \eqref{est: chaos large lambda} into \eqref{est: chaos mid step}, we have thus obtained that
\begin{align*}
%&\int_{S_{t,k}}dw\,\prod_{i=1}^k\int_{\mathbb{R}}|\lambda_i|^n \sum_{m_i\in\mathbb{N}^n} |\lambda_i|^{-2\alpha}(2|m_i|+n)^{-2\alpha}e^{-8w_i|\lambda_i|(2|m_i|+n)}d\lambda_i\\
M_{n,k}\leq \sum_{I\subset\{1,...,k\}} \frac{(tC_1D^-_N)^{|I|}}{|I|!}\left({C_2D^+_N}\right)^{|I^c|}=\sum_{j=1}^k{k\choose j}  \frac{(tC_1D^-_N)^{j}}{j!}\left({C_2D^+_N}\right)^{k-j}.
\end{align*}
The proof of our claim \eqref{final bound M_nk} is thus completed.
\end{proof}

\subsubsection{Existence, uniqueness and moment estimates}\label{sec:existence-chaos-non-smooth} 
This section is devoted to prove our existence-uniqueness result thanks to chaos expansion. This is summarized in the theorem below. Notice that an advantage of the chaos expansion method is that it produces some necessary and sufficient condition on the parameter $\alpha$.

\begin{theorem}\label{thm:existence-uniqueness chaos}
Assume $\bW_\alpha$ is a Gaussian noise whose covariance function is given by~\eqref{eq-d}. Then if $\alpha\in (\frac{n}{2}, \frac{n+1}{2})$, equation \eqref{mild sol} admits a unique solution in the It\^{o}-Skorohod sense (as given in Proposition~\ref{general existence chaos}). Moreover provided $|u_0(x)|\geq\kappa$ for a constant $\kappa>0$, the condition $\alpha\in (\frac{n}{2}, \frac{n+1}{2})$ is also necessary in order to solve equation  \eqref{mild sol}.
\end{theorem}
\begin{proof}
We divide this proof into sufficient and necessary conditions.

\noindent
\textit{Step 1: Sufficient condition.}
Let us assume that $\alpha\in (\frac{n}{2}, \frac{n+1}{2})$.
As stated in Proposition~\ref{general existence chaos}, we just need to check inequality~\eqref{exist of sol-chaos expansion} in order to get existence and uniqueness. Now gathering~\eqref{final bound M_nk} and~\eqref{chaos bound}, we get
\begin{equation}
\|f_k(\cdot,t,q)\|_{\mathcal{H}^{\otimes k}}^2
\leq \frac{C_0^k}{k!}\sum_{j=1}^k{k\choose j}  \frac{(tC_1D^-_N)^{j}}{j!}\left({C_2D^+_N}\right)^{k-j},
\end{equation}
where we have set $C_0={2^{n-1}}{\pi^{-(n+1)}}$. Summing over $k$ and using the elementary inequality  ${k\choose j}\leq 2^k$ we obtain
%With the help of Lemma \ref{est: chaos Fourier mode}, combining \eqref{f_n explicit}, \eqref{H-norm g} and \eqref{Chaos Fourier mode}, and using the elementary inequality ${k\choose j}\leq 2^k$ we obtain
\begin{align}
\sum_{k=0}^\infty k!\|f_k(\cdot,t,q)\|_{\mathcal{H}^{\otimes k}}^2&\leq \|u_0\|^2_\infty \sum_{k=0}^\infty C_0^k\sum_{j=1}^k{k\choose j}  \frac{(tC_1D^-_N)^{j}}{j!}\left({C_2D^+_N}\right)^{k-j}\nonumber\\
&\leq \|u_0\|^2_\infty \sum_{k=0}^\infty \sum_{j=1}^k (2C_0)^k  \frac{(tC_1D^-_N)^{j}}{j!}\left({C_2D^+_N}\right)^{k-j}\nonumber\\
&= \|u_0\|^2_\infty  \sum_{j=1}^\infty   \frac{(tC_1D^-_N)^{j}}{j!} \left({C_2D^+_N}\right)^{-j} \sum_{k=j}^\infty (2C_0)^k  \left({C_2D^+_N}\right)^{k}.\label{eq: final est mid step}
\end{align}
Now we can choose $N$ large enough to ensure the summation over $k$ in \eqref{eq: final est mid step} converges and equals to
$\frac{(2C_0C_2D^+_N)^j}{1-2C_0C_2D^+_N}$ under our standing assumption on $\alpha$.
Plugging this into \eqref{eq: final est mid step} gives us
\begin{align}\label{existence last step bound}
\sum_{k=0}^\infty k!\|f_k(\cdot,t,q)\|_{\mathcal{H}^{\otimes k}}^2\leq  \frac{ \|u_0\|^2_\infty }{1-2C_0C_2D^+_N}\sum_{j=1}^\infty   \frac{(tC_1D^-_N)^{j}}{j!} \left({C_2D^+_N}\right)^{-j} (2C_0C_2D^+_N)^j<\infty.
\end{align}
This proves \eqref{exist of sol-chaos expansion} and hence \eqref{mild sol} admits a unique solution. 

\noindent
\textit{Step 2: Necessary condition.} If the initial condition $u_0$ is such that $u_0(x)\geq\kappa>0$, then according to \eqref{chaos bound} and \eqref{Chaos Fourier mode1} a necessary condition in order to solve our heat equation \eqref{eq:pam} is that $M_{n,1}<\infty$. Moreover, after some elementary simplifications we discover that
$$M_{n,1}=2\sum_{m\in\mathbb{N}^n}(2|m|+n)^{-(2\alpha+1)}\int_0^\infty\lambda^{n-2\alpha-1}\left(1-e^{-8\lambda(2|m|+n)}\right)d\lambda.$$
Now considering the term $m=0_{\mathbb{N}^n}$ only we get
\[M_{n,1}\gtrsim\int_0^\infty\lambda^{n-2\alpha-1}\left(1-e^{-8n\lambda}\right)d\lambda.\]
Then we divide $[0,\infty)$ in to $[0,1)\cup[1,\infty)$ and use the bound $1-e^{-x}\gtrsim x$ on $[0,1)$ and $1-e^{-x}\leq 1$ on $[1,\infty)$. This yields
\begin{align}
M_{n,1}\gtrsim\int_0^1\lambda^{n-2\alpha}d\lambda+\int_1^\infty \lambda^{n-2\alpha-1}d\lambda,\label{necessity}.
\end{align}
From the right hand-side of \eqref{necessity}, it is readily checked that the condition $\alpha\in(\frac{n}{2},\frac{n+1}{2})$ is necessary for the finiteness of $M_{n,1}$. 

Step 1 and Step 2 complete the proof of existence and uniqueness of solution to equation~\eqref{eq:pam} by the method of chaos expansion.
\end{proof}

We now state a rough exponential estimate for the $L^2$-moments of $u(t,q)$. This non optimal bound illustrates how useful the chaos approach can be in our context. 

\begin{proposition}\label{thm: exponential L^2 bound}
As in Theorem \ref{thm:existence-uniqueness chaos}, we consider the noise $\bW_\alpha$ with $\alpha\in(\frac{n}{2},\frac{n+1}{2})$. Let $u$ be the unique solution to equation \eqref{mild sol}. Then there exist $c_1,..., c_4>0$ such that for all $(t,q)\in\R_+\times\Heis$ we have
\begin{align}\label{exponential L^2 bound}
c_1e^{c_2t}\leq \E\left[(u(t,q))^2\right]\leq c_3e^{c_4t}.
\end{align}
\end{proposition}

Before proving Proposition \ref{thm: exponential L^2 bound}, we need to state a couple of preliminary results on the Brownian motion $\cb$ on $\Heis$. 
We first label a basic scaling result for this process.

\begin{proposition}\label{prop:scaling-bm}
Let $\mathcal{B}$ be the Brownian motion on $\Heis$ introduced in Section \ref{intro BM on Heis}, and recall that the distance $d_{cc}$ is given by \eqref{eq-cc-dist}. We assume that $\mathcal{B}_0=e$ almost surely. Then for $t>0$ the following scaling property holds true:
\begin{equation}
d_{cc}(\mathcal{B}_t, e)\stackrel{(\mathcal{D})}{=}\sqrt{t} \, d_{cc}(\mathcal{B}_1,e)\label{scaling property BM}.
\end{equation}
\end{proposition}
\begin{proof}
For the dilation $\delta_{\la}$ given in~\eqref{eq-dil-delta}, it is shown in \cite[pp. 36-37]{Ba} that $\cb_{t}\stackrel{(\mathcal{D})}{=}\delta_{\sqrt{t}}\,\cb_{1}$. Hence the result follows from the fact that $d_{cc}(\delta_{\sqrt{t}}\,\mathcal{B}_1,e)=\sqrt{t} \, d_{cc}(\mathcal{B}_1,e)$.
\end{proof}

We also label an estimate about small ball probabilities for the Brownian motion on $\Heis$, which will be crucial for our lower bound on $L^{2}$-moments of the stochastic heat equation.

\begin{lemma} Let $\mathcal{B}$ be the Brownian motion on $\Heis$ defined in Section \ref{intro BM on Heis}. We also consider an independent copy $\tilde{\mathcal{B}}$ of $\mathcal{B}$. For $\epsilon, t>0$ such that $\epsilon\ll \sqrt{t}$, we define an event 
\begin{align}\label{A_ep,t}
A_{\epsilon,t}=\left\{\sup_{0\leq s\leq t}d_{cc}(\mathcal{B}_s,\tilde{\mathcal{B}}_s)\leq \epsilon\right\}.
\end{align}
Then there exist two constants $c, C>0$ such that
\begin{align}\label{small ball proba}
\mathbb{P}(A_{\epsilon,t}) \ge c \exp\left\{-\frac{Ct}{\epsilon^2}\right\}.
\end{align}
\end{lemma}

\begin{proof} We assume (without loss of generality) that $\mathcal{B}_0=\tilde{\mathcal{B}}_0=e$. Then it is easily seen that
\begin{align*}
\left\{\sup_{0\leq s\leq t}d_{cc}(\mathcal{B}_s,\tilde{\mathcal{B}}_s)\leq \epsilon \right\}\supset\left\{ \sup_{0\leq s\leq t}d_{cc}(\mathcal{B}_s,x)\leq \frac{\epsilon}{2} \right\} \cap \left\{ \sup_{0\leq s\leq t}d_{cc}(x,\tilde{\mathcal{B}}_s)\leq \frac{\epsilon}{2}\right\}.
\end{align*}
Therefore invoking the independence of $\mathcal{B}$ and $\tilde{\mathcal{B}}$ we obtain
\begin{align}
\mathbb P \left(\sup_{0\leq s\leq t}d_{cc}(\mathcal{B}_s,\tilde{\mathcal{B}}_s)\leq \epsilon \right) 
&  \ge 
\mathbb P \left(\left\{ \sup_{0\leq s\leq t}d_{cc}(\mathcal{B}_s,e)\leq \frac{\epsilon}{2} \right\} 
\cap 
\left\{ \sup_{0\leq s\leq t}d_{cc}(e,\tilde{\mathcal{B}}_s)\leq \frac{\epsilon}{2}\right\}\right) \notag\\
 &\ge \mathbb P \left( \sup_{0\leq s\leq t}d_{cc}(\mathcal{B}_s,e)\leq \frac{\epsilon}{2} \right)^2.\label{lower proba bound}
\end{align}
We now bound the right hand-side of \eqref{lower proba bound}. To this aim, for $r>0$, let us denote by $T_{r}$ the hitting time by $\mathcal{B}$ of the Carnot-Carath\'eodory sphere with center $x$ and radius $r$. Using the scaling property for the process $s\mapsto d_{cc}(\mathcal{B}_s,e)$ stated in Proposition~\ref{prop:scaling-bm}, we can then write
\[
 \mathbb P \left( \sup_{0\leq s\leq t}d_{cc}(\mathcal{B}_s,x)\leq \frac{\epsilon}{2} \right) \ge \mathbb{P} \left( T_{\epsilon /2} \ge t \right)=\mathbb{P} \left( T_{1} \ge \frac{4t}{\varepsilon^2} \right).
\]
Moreover, according to Theorem 5.2 (see also Corollary 5.4) in  \cite{BGNT}, one has
\[
\mathbb{P} \left( T_{1} \ge \frac{4t}{\varepsilon^2} \right) \ge \mathbb{P} \left( \tau_{1} \ge \frac{4t}{\varepsilon^2} \right), 
\]
where $\tau_{1}$ is the hitting time of 1 by a Bessel process of dimension $2n+3$ started from zero. The result follows then from classical estimates.
\end{proof}

\begin{remark}
In fact the proof above also yields a more precise estimate. Namely for every $\epsilon >0$ we have
\[
\liminf_{t \to +\infty} \frac{1}{t}  \ln  \mathbb{P}(A_{\epsilon,t}) \ge -4 \frac{\lambda_1}{\epsilon^2},
\]
where $\lambda_1$ is the first Dirichlet eigenvalue of the unit ball in $\mathbb R^{2n+3}$.
\end{remark}

We now turn to the proof of our estimate for $L^2$-moments of $u(t,q)$. It relies on a Feynman-Kac representation of moments which will be mostly detailed in section \ref{sec-smooth} (the Feynman-Kac representation being crucial in the smooth noise case).

\begin{proof}[Proof of Proposition \ref{thm: exponential L^2 bound}]
The upper bound in \eqref{exponential L^2 bound} is an easy consequence of \eqref{existence last step bound}. We thus focus on the lower bound. Towards this aim, we start by stating a Feynman-Kac representation for the $L^2$-moments. This claim will be detailed below in the proof of Theorem~ \ref{thm:existence-smooth-noise}. Here we content ourselves with asserting that
\begin{align}\label{eq-L-2-mom}
\mathbf{E}\left[u_t(x)^2\right]
= \mathbb{E} \lc \exp\left\{c\int_0^tG_{2\alpha}(\mathcal{B}_s,\tilde{\mathcal{B}}_s)ds\right\}\rc,
\end{align}
where $\mathcal{B}, \tilde{\mathcal{B}}$ are two independent Brownian motions on $\Heis$ starting from $e$, and where $\mathbb{E}$ designates the expected value with respect to $\mathcal{B}$, $\tilde{\mathcal{B}}$ only. Now for $\epsilon, t>0$ consider the event $A_{\epsilon,t}$ defined by \eqref{A_ep,t}. Due to the positivity of the exponential function we get
\begin{align*}
\E\left[(u(t,q))^2\right]\gtrsim \mathbb{E} \left[\exp\left\{c\int_0^tG_{2\alpha}(\mathcal{B}_s,\tilde{\mathcal{B}}_s)ds\right\}\mathbf{1}_{A_{\epsilon,t}}\right].
%&\geq \exp\left\{\frac{ct}{\epsilon^{2(n+1-2\alpha)}}\right\}\mathbb{P}(A_{\epsilon,t}).
\end{align*}
Moreover, thanks to the definition of $A_{\epsilon,t}$ and resorting to \eqref{eq-kernel-cc} in order to bound $G_{2\alpha}$ we obtain 
\begin{align*}
\E\left[(u(t,q))^2\right]\gtrsim \exp\left\{\frac{ct}{\epsilon^{2(n+1-2\alpha)}}\right\}\mathbb{P}(A_{\epsilon,t}).
\end{align*}
Invoking relation \eqref{small ball proba} for $\epsilon\ll\sqrt{t}$, this yields
\begin{align}\label{picking ep large 1}
\mathbf{E}[u_t(x)^2]\gtrsim\exp\left\{\left(\frac{c}{\epsilon^{2(n+1-2\alpha)}}-\frac{C}{\epsilon^2}\right)t\right\}.
\end{align}
Note that whenever $\alpha<\frac{n+1}{2}$, we have $2(n+1-2\alpha)<2$. Hence if $t$ is large enough, one can find $\epsilon$ large enough but still with $\epsilon\ll\sqrt{t}$ such that
\begin{align}\label{picking ep large 2}
\frac{c}{\epsilon^{2(n+1-2\alpha)}}-\frac{C}{\epsilon^2}\geq \frac{c}{2\epsilon^{2(n+1-2\alpha)}}\,.
\end{align}
Plugging \eqref{picking ep large 2} into \eqref{picking ep large 1}, the lower bound in \eqref{exponential L^2 bound} is proved.
\end{proof}

\subsubsection{Smooth noise regime}\label{sec-smooth} We now turn to the analysis of the mild equation \eqref{mild sol} with chaos expansions, in the smooth case $\alpha>\frac{n+1}{2}$. Notice that this case is not included in Theorem \ref{prop:ito-exist-unique} above. Generally speaking developments for SPDEs driven by function-valued noises are scarce in the literature. Moreover covariances for noises of the form \eqref{eq-W-ptwise} are growing polynomially at $\infty$, which induces some additional technical difficulty in the analysis of stochastic convolutions like \eqref{stochastic integral}. Below we will see how to solve this obstacle thanks to a Feynman-Kac representation of moments.

\begin{theorem}\label{thm:existence-smooth-noise}
Consider the Gaussian noise $\bW_\alpha$ for $\frac{n+1}{2}<\alpha<\frac{n+2}{2}$, as introduced in Proposition \ref{prop-W-pt}. Then equation \eqref{mild sol}        
interpreted in the It\^{o}-Skorohod sense admits a unique solution.
\end{theorem}
\begin{proof}This proof is not a direct application of the chaos estimate \eqref{exist of sol-chaos expansion}. It is based instead on a closely related Feynman-Kac formula already alluded to in \eqref{eq-L-2-mom}. We divide it in several steps.

\noindent{\it Step 1: Reduction to an exponential estimate.}  Owing to relation \eqref{eq-W-ptwise}, the covariance function of $\bW_\alpha$ is given by
\[\mathbf{E}\left[\bW_\alpha(s,q)\bW_\alpha(t,q')\right]=(s\wedge t)\gamma(q,q'),\]
where the function $\gamma$ is defined by
\begin{align}\label{covariance smooth case}
\gamma(q,q'):=\int_{\mathbf H^n} (G_\alpha (q,r)-G_\alpha (e,r))(G_\alpha (q',r)-G_\alpha (e,r)) dr.
\end{align}
We now claim that the existence of a solution to \eqref{mild sol} can be reduced to prove that for all $\beta>0$ we have
\begin{align}\label{criterion Feyman-Kac}
\mathbb{E}_x\left[ \exp\left({\beta\int_0^t\gamma(\mathcal{B}_s,\tilde{\mathcal{B}}_s)ds}\right)\right]<\infty,
\end{align}
where $\mathcal{B}_s$ and $\tilde{\mathcal{B}}_s$ are two independent Brownian motions on $\Heis$ starting from the same point $x$. In \eqref{criterion Feyman-Kac}, also notice that $\mathbb{E}_x$ designates the expected value with respect to the randomness in $\mathcal{B}$ and $\tilde{\mathcal{B}}$.

In order to prove that existence of a solution to \eqref{mild sol} amounts to \eqref{criterion Feyman-Kac}, we start from~\eqref{exist of sol-chaos expansion}. Indeed, \eqref{exist of sol-chaos expansion} asserts that a precise $L^2(\Omega)$-bound on the solution $u$ yields existence and uniqueness. Next we consider a smooth noise 
\[ \bW_\alpha^\epsilon=\bW_\alpha*p_\epsilon,\]
where $p_\epsilon$ is the heat kernel in \eqref{eq-Heis-kernel}. Then one can show that equation \eqref{mild sol} driven by the (smooth in space) noise $\bW_\alpha^\epsilon$ admits a unique solution called $u^\epsilon$. Furthermore similarly to \cite[Proposition 4.7]{HHNT}, relation \eqref{criterion Feyman-Kac} ensures that $u^\epsilon$ converges in $L^2$ to $u$, where $u$ solves~\eqref{mild sol}. Our aim will thus be to prove \eqref{criterion Feyman-Kac}. Notice that this method amounts to using a Feynman-Kac representation of the solution.

\noindent{\it Step 2: Estimates on $\gamma$.} Recall that Proposition \ref{prop-integrability-G} states that for all $q\in\Heis$ we have
\[
\int_{\Heis} (G_\alpha (q,r)-G_\alpha (e,r))^2 dr \le Cd_{cc}(q,e)^{4\alpha-2(n+1)}.
\]
Plugging this relation into \eqref{covariance smooth case} and applying the elementary inequality $|ab|\leq (a^2+b^2)/2$, we get
\begin{align}\label{esti covariance smooth case}
|\gamma(q,q')|\lesssim\left( d_{cc}(q,e)^{4\alpha-2(n+1)}+ d_{cc}(q',e)^{4\alpha+2(n+1)}\right).
\end{align}
Reporting \eqref{esti covariance smooth case} into \eqref{criterion Feyman-Kac} and invoking Schwarz' inequality, it is now sufficient to prove that for all $\beta>0$ we have
\begin{align}\label{smooth case sufficient cond}
\mathbb{E} \left[\exp\left(\beta\int_0^td_{cc}(\mathcal{B}_s,e)^{4\alpha-2(n+1)}ds\right)\right]<\infty.
\end{align}
In the sequel we turn our attention to prove \eqref{smooth case sufficient cond}.

\noindent{\it Step 3: Proof of \eqref{smooth case sufficient cond}.} We apply Jensen's inequality to the left hand-side of \eqref{smooth case sufficient cond}. This yields
\begin{align}
\mathbb{E} \lc \exp\lp \beta\int_0^td_{cc}(\mathcal{B}_s,e)^{4\alpha-2(n+1)}ds\rp \rc
&=\mathbb{E} \lc \exp\lp \beta t\int_0^td_{cc}(\mathcal{B}_s,e)^{4\alpha-2(n+1)}\frac{ds}{t} \rp\rc\notag\\
&\leq \frac{1}{t}\mathbb{E}\lc \int_0^t \exp\lp \beta t\, d_{cc}(\mathcal{B}_s,e)^{4\alpha-2(n+1)}\rp ds\rc\notag\\
&=\frac{1}{t}\int_0^t \mathbb{E}\lc \exp\lp\beta t \,d_{cc}(\mathcal{B}_s,e)^{4\alpha-2(n+1)}\rp\rc ds.\label{needs to be done}
\end{align}
Owing to \eqref{eq-heat-kernel-bds}, it is now readily seen that the right hand-side of \eqref{needs to be done} is finite as long as  $4\alpha-2(n+1)<2$ (or otherwise stated $\alpha<\frac{n+2}{2}$).

\noindent{\it Step 4: Conclusion.} We have proved that \eqref{criterion Feyman-Kac} holds true as long as $\alpha<\frac{n+2}{2}$. According to our considerations in Step 1, we have thus obtained the existence  and uniqueness of a solution to \eqref{mild sol} under this condition on $\alpha$. 
\end{proof}

Interestingly enough (and similarly to the distributional noise case of Section~\ref{sec:existence-chaos-non-smooth}), the Feynman-Kac type methodology put forward in Theorem~\ref{thm:existence-smooth-noise} also yields asymptotic results for the second moment of $u$ through Laplace method. Below is a result in this direction. 

\begin{theorem} In the regime $\frac{n+1}{2}<\alpha<\frac{n+2}{2}$ of Theorem \ref{thm:existence-smooth-noise}, let $u$ be the unique  solution to equation \eqref{mild sol}. Then for all $q\in\Heis$, $u_t(q)$ admits the following asymptotic upper bound for its second moment: there exist two constants $C_1, C_2>0$ such that for $t$ large enough we have
\begin{align}\label{L2 upper bound smooth case}
\mathbf{E} [u_t(x)^2]\leq  C_1 {t^{\rho n}}\exp\left\{ C_2t^{\rho}\right\},
\end{align}
where the exponent $\rho=\rho(\al)$ satisfies
\begin{equation}\label{eq:def-nu}
1< \rho =\frac{2\al-n}{2-(2\al-n)} <\infty ,
\quad\text{as}\quad 
\frac{n+1}{2}<\alpha<\frac{n+2}{2}.
\end{equation}
\end{theorem}

\begin{proof} According to equations \eqref{criterion Feyman-Kac} and \eqref{needs to be done}, the second moment of $u_t(q)$ is controlled by a quantity $Q_t$ of the form
\begin{equation}\label{asyp L2 upper bound smooth}
Q_t=\frac{1}{t}\int_0^t \mathbb{E}\left[ \exp\left({{C}t \,d_{cc}(\mathcal{B}_s,e)^{4\alpha-2(n+1)}}\right)\right] ds,
\end{equation}
for a constant $C>0$. We will now estimate the right hand-side of  \eqref{asyp L2 upper bound smooth}. Apply first the scaling property \eqref{scaling property BM} to the right hand-side of \eqref{asyp L2 upper bound smooth}. We get
\[
Q_t=\frac{1}{t}\int_0^t \mathbb{E}\left[\exp\left({{C}ts^{2\alpha-(n+1)} \,d_{cc}(\mathcal{B}_1,e)^{4\alpha-2(n+1)}}\right)\right]ds.
\]
Moreover recall that we are considering the regime $\alpha>\frac{n+1}{2}$. Hence bounding $s^{2\alpha-(n+1)}$ by $t^{2\alpha-(n+1)}$, we end up with
\[
Q_t \le  \mathbb{E}\left[\exp\left({{C}t^{2\alpha-n} \,d_{cc}(\mathcal{B}_1,e)^{4\alpha-2(n+1)}}\right)\right].
\]
Therefore invoking \eqref{eq-heat-kernel-bds} we obtain
\[
Q_t \le C_3\int_\Heis  \exp\left({{C}t^{2\alpha -n}\, d_{cc}(q,e)^{4\alpha-2(n+1)}-C_4 \, d(e,q)^{2}}\right) d\mu(q).
\]
We now resort to polar coordinates, similarly to what we did after equation \eqref{eq-soln-ub}. This allows to write 
\begin{equation}\label{eq-Laplace-est}
Q_t\lesssim \int_0^{+\infty} r^{2n+1} \exp\left({{C} t^{2\alpha-n}\, r^{4\alpha-2(n+1)}-C_4 \, r^{2}}\right)dr.
\end{equation}
We are thus reduced to evaluate the $1$-d integral in the right hand-side of \eqref{eq-Laplace-est}.

The right hand-side of \eqref{eq-Laplace-est} will be bounded thanks to Laplace asymptotics. To this aim, consider the change of variable 
$r=t^{\nu}u$, where $\nu=\rho/2$ and $\rho$ is defined by \eqref{eq:def-nu}. It is elementary to check that 
$$t^{2\alpha-n}t^{\nu(4\alpha-2(n+1))}=t^{2\nu}.$$ 
Hence the change of variable in \eqref{eq-Laplace-est} leads to the upper bound 
\begin{align}\label{after change of variable}
Q_t\lesssim {t^{\nu(2n+1)}}\int_0^{+\infty} u^{2n+1} e^{t^{2\nu}f(u)}du
\end{align}
where we have set
\[
f(u)={C}\, u^{2\alpha-n}-{C_{4}} \, u^{2}.
\] 
Recall again that $2\alpha-n<2$ under our set of assumptions. Hence one can easily optimize $f$ over $\mathbb{R}$. We get
\begin{align}\label{u*}
u^*&\equiv\arg \min f(u)=\lp\frac{C(2\alpha-n)}{2C_4}\rp^{\frac{1}{1-2\alpha+(n+1)}},\\
f(u^*)&= \frac{C}{4}(2-2\alpha+n) \, (u^*)^{2\alpha-n} .\label{f(u*)}
\end{align}
We also let the patient reader check that
\begin{align}\label{f''}
f''(u^*)=2C_4(2\alpha-(n+2))<0.
\end{align}
According to Laplace's method (see e.g. \cite[Section 6.4]{BO}) and going back to  \eqref{after change of variable}, for $t$ large enough we have
\[
Q_t\lesssim t^{(2n+1)\nu}(u^*)^{2n+1}e^{t^{2\nu}f(u^*)}\left({\frac{2\pi}{t^{2\nu}|f''(u^*)|}}\right)^{\frac{1}{2}}.
\]
Reporting the values \eqref{u*}, \eqref{f(u*)} and \eqref{f''} in the inequality above, we can thus write
\begin{align*}
%&\frac{1}{t}\int_0^t \mathbb{E}e^{\frac{C}{2}t \,d(B_s,e)^{4\alpha-2(n+1)}} ds
%\le \int_0^{+\infty} r^{2n+1} e^{\frac{C}{2} t^{2\alpha-n}\, r^{4\alpha-2(n+1)}-c_4 \, r^{2}}dr \\
Q_t\lesssim{t^{2\nu n}}\exp\left\{ C't^{2\nu}\right\}=t^{n\rho} \exp\left\{ C't^{\rho}\right\},
\end{align*}
where we recall that $\rho=2\nu$. This proves  our claim \eqref{L2 upper bound smooth case}.
\end{proof}

\bigskip

\paragraph{\textbf{Acknowledgment.}}
We would like to thank Jean-Yves Chemin for some interesting discussion about the Fourier transform on Heisenberg groups.

\bibliographystyle{amsplain}
\bibliography{biblio}

\end{document}